\PassOptionsToPackage{hyphens}{url}
\PassOptionsToPackage{hypertexnames=false}{hyperref}
\documentclass[a4paper, USenglish, numberwithinsect, cleveref, autoref, thm-restate]{lipics-v2021}

\usepackage{microtype}
\usepackage{csquotes}
\usepackage{wrapfig}

\crefname{conjecture}{conjecture}{conjectures}
\crefname{observation}{observation}{observations}

\pdfoutput=1
\hideLIPIcs

\title{Towards Crossing-Free Hamiltonian Cycles in Simple Drawings of Complete Graphs}

\author{Oswin Aichholzer}{Institute of Software Technology, Graz University of Technology, Austria}{oaich@ist.tugraz.at}{https://orcid.org/0000-0002-2364-0583}{Partially supported by the Austrian Science Fund (FWF) grant W1230}

\author{Joachim Orthaber}{Institute of Software Technology, Graz University of Technology, Austria}{orthaber@ist.tugraz.at}{https://orcid.org/0000-0002-9982-0070}{Supported by the Austrian Science Fund (FWF) grant W1230}

\author{Birgit Vogtenhuber}{Institute of Software Technology, Graz University of Technology, Austria}{bvogt@ist.tugraz.at}{https://orcid.org/0000-0002-7166-4467}{Partially supported by Austrian Science Fund within the collaborative DACH project \emph{Arrangements and Drawings} as FWF project \mbox{I 3340-N35}}

\authorrunning{O. Aichholzer, J. Orthaber, and B. Vogtenhuber}

\Copyright{Oswin Aichholzer, Joachim Orthaber, and Birgit Vogtenhuber}

\ccsdesc[500]{Theory of computation~Computational geometry}

\keywords{crossing-free Hamiltonian cycles and paths, classes of simple drawings}

\relatedversion{~}

\relatedversiondetails[linktext={https://graz.elsevierpure.com/...-of-the-comple}]{Master's thesis of Joachim Orthaber}{https://graz.elsevierpure.com/en/publications/crossing-free-hamiltonian-cycles-in-simple-drawings-of-the-comple}
\relatedversiondetails[linktext={https://dccg.upc.edu/eurocg23/.../Session-5B-Talk-2.pdf}]{EuroCG version}{https://dccg.upc.edu/eurocg23/wp-content/uploads/2023/03/Session-5B-Talk-2.pdf}
\relatedversiondetails[linktext={https://doi.org/10.57717/cgt.v3i2.47}]{CGT paper}{https://doi.org/10.57717/cgt.v3i2.47}

\acknowledgements{We thank Rosna Paul, Daniel Perz, and Alexandra Weinberger for fruitful discussions. We further thank Manfred Scheucher for his computations and we are grateful for the colors recommended by Wong~\cite{w-2011-cb}, which we use in our figures. Last but not least we thank the anonymous referees for their valuable comments.}

\nolinenumbers

\begin{document}

\maketitle

\begin{abstract}
It is a longstanding conjecture that every simple drawing of a complete graph on $n\geq 3$ vertices contains a crossing-free Hamiltonian cycle. We strengthen this conjecture to \enquote{there exists a crossing-free Hamiltonian path between each pair of vertices} and show that this stronger conjecture holds for several classes of simple drawings, including strongly c-monotone drawings and cylindrical drawings. As a second main contribution, we give an overview on different classes of simple drawings and investigate inclusion relations between them up to weak isomorphism.
\end{abstract}

\section{Introduction}\label{sec:intro}

In a \emph{drawing} of a graph in the plane (or on the sphere) vertices are represented by distinct points and edges by Jordan arcs connecting the respective points. For convenience we also call those representations vertices and edges. A \emph{simple drawing} is a drawing of a graph where each pair of edges meets in at most one point, which has to be either a proper crossing or a common end-vertex. A fundamental line of research in this area is concerned with finding \emph{crossing-free} sub-drawings in simple drawings of the complete graph $K_{n}$ on $n$ vertices. These are sub-drawings with pairwise non-crossing edges, also called \emph{plane} sub-drawings. For this task it is sufficient to know all crossings in the drawing, its precise shape is negligible. We call two simple drawings $\mathcal{D}$ and $\mathcal{D}'$ of the same graph \emph{weakly isomorphic} if two edges in~$\mathcal{D}$ cross if and only if the corresponding edges in $\mathcal{D}'$ cross. We call them \emph{strongly isomorphic} if there exists a homeomorphism (of the sphere) that maps $\mathcal{D}$ to $\mathcal{D}'$. Weak isomorphism classes of~$K_{n}$ can be uniquely represented by rotation systems. The \emph{rotation} of a vertex $v$ is the (clockwise) circular order of all edges that are incident to the vertex $v$, where these edges are classically identified by their second end-vertex. The collection of the rotations of all vertices in a simple drawing is called its \emph{rotation system}. See {\'A}brego et al.~\cite{aafhpprsv-2015-agdscg} and Kyn{\v{c}}l~\cite{k-2009-esctg} for more information on rotation systems and see \Cref{fig:five-K5s} for all non-isomorphic simple drawings of $K_{5}$. In $1988$, Nabil Rafla stated the following conjecture in his PhD thesis~\cite{r-1988-gdcg}.

\begin{conjecture}[Rafla~\cite{r-1988-gdcg}]\label{conj:main}
Every simple drawing of the complete graph $K_{n}$ on $n \geq 3$ vertices contains at least one crossing-free Hamiltonian cycle.
\end{conjecture}

\subparagraph{Related work.}

Under the assumption that \Cref{conj:main} is true, Rafla enumerated all different simple drawings of $K_{n}$ for $n \leq 7$ up to weak isomorphism. Since then, \Cref{conj:main} and relaxations of it have attracted considerable attention. Especially, a crossing-free Hamiltonian cycle in a simple drawing~$\mathcal{D}$ implies, by removing an arbitrary edge of the cycle, that $\mathcal{D}$ also contains a crossing-free Hamiltonian path. Furthermore, for even $n$, a crossing-free Hamiltonian path in turn implies that $\mathcal{D}$ contains a plane perfect matching, by disregarding every second edge of the path. However, even the question of the existence of a plane perfect matching in every simple drawing of $K_{n}$, for even $n$, is still open.

In $2003$, Pach, Solymosi, and T{\'o}th~\cite{pst-2003-ucctg} showed that every simple drawing of $K_{n}$ contains plane sub-drawings isomorphic to any tree of size $\mathcal{O}(\log(n)^{1/6})$. This immediately implies a lower bound of $\Omega(\log(n)^{1/6})$ for the longest crossing-free path and largest plane matching in every simple drawing of $K_{n}$. Subsequently, a lot of progress has been made with regard to plane matchings. Until recently, a lower bound of $\Omega(n^{1/2-\varepsilon})$ for any fixed $\varepsilon > 0$, shown by Ruiz-Vargas~\cite{r-2017-mdetg}, was best known. This bound has lately been improved to $\Omega(\sqrt{n})$ in~\cite{agtvw-2022-twfpssdcg}, via the introduction and use of generalized twisted drawings; see \Cref{sec:classes} for a definition of generalized twisted drawings and several other drawing classes mentioned in this introduction. In the same paper, and independently also by Suk and Zeng~\cite{sz-2022-upcstg}, a lower bound of $\Omega(\log(n)/\log(\log(n)))$ for the longest crossing-free path was shown. This is the first improvement for crossing-free paths in almost $20$ years since the result from Pach, Solymosi, and T{\'o}th. Furthermore, the authors of \cite{agtvw-2022-twfpssdcg} obtained the same bound of $\Omega(\log(n)/\log(\log(n)))$ for the longest crossing-free cycle. For a more detailed history on the developments during the last $20$ years, we point the interested reader to the references given in~\cite{agtvw-2022-twfpssdcg,r-2017-mdetg,sz-2022-upcstg}.

In a similar direction, Pach, Solymosi, and T{\'o}th~\cite{pst-2003-ucctg} showed that every simple drawing of $K_{n}$ contains a sub-drawing of size $\Omega(\log(n)^{1/8})$ that is weakly isomorphic to either a convex straight-line drawing or a so-called twisted drawing. This implies the existence of various plane sub-drawings of the respective size. Recently, Suk and Zeng~\cite{sz-2022-upcstg} improved this $\Omega(\log(n)^{1/8})$ bound to $\Omega(\log(n)^{1/4-\varepsilon})$ for any fixed $\varepsilon > 0$. This also improves the above mentioned $\mathcal{O}(\log(n)^{1/6})$ bound for trees accordingly.

Furthermore, Ruiz-Vargas~\cite{r-2017-mdetg} showed that every c-monotone drawing of $K_{n}$ contains a plane matching of size $\Omega(n^{1-\epsilon})$ for any fixed $\varepsilon > 0$; so \enquote{almost} a perfect matching. In addition, in \cite{agtvw-2022-twfpssdcg} it was shown that every c-monotone drawing contains a sub-drawing of size $\Omega(\sqrt{n})$ that is weakly isomorphic to either an $x$-monotone drawing or a generalized twisted drawing, implying that c-monotone drawings of $K_{n}$ contain a crossing-free path as well as a crossing-free cycle of size $\Omega(\sqrt{n})$.

Concerning crossing-free Hamiltonian cycles, \Cref{conj:main} has been confirmed for all simple drawings on $n \leq 9$ vertices using the rotation system database~\cite{aafhpprsv-2015-agdscg}, and Ebenführer tested the conjecture  on randomly generated realizable rotation systems for up to $30$ vertices in his Master's thesis~\cite{e-2017-rrs}. Moreover, Arroyo, Richter, and Sunohara~\cite{ars-2021-edcgap} proved the existence of a crossing-free Hamiltonian cycle in so-called pseudospherical (or h-convex) drawings of~$K_{n}$. In a current paper, Bergold, Felsner, M.\ Reddy, and Scheucher~\cite{bfrs-2023-usspssd} extended this to (generalized) convex drawings. Moreover, they confirmed \Cref{conj:main} for $n=10$ using a SAT encoding.
\begin{wrapfigure}[13]{r}{.45\textwidth}
\centering
\includegraphics[page=1]{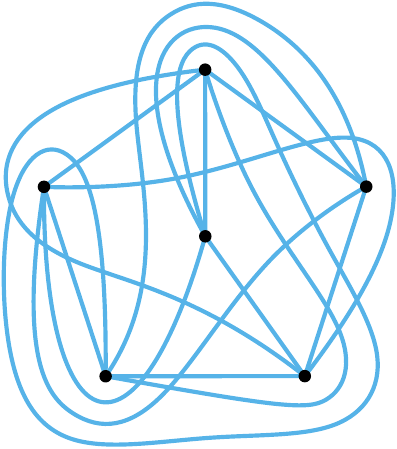}
\caption{\textbf{(\cite{aeppv-2017-ssdcg,e-2017-rrs})} A star-simple drawing that has no crossing-free Hamiltonian cycle.}
\label{fig:star-simple}
\end{wrapfigure}
Finally, in \cite{agtvw-2022-twfpssdcg}, \Cref{conj:main} was shown to be true for generalized twisted drawings on an odd number of vertices.

On the negative side, in~\cite{aeppv-2017-ssdcg,e-2017-rrs} the authors showed that simplicity of the drawings is crucial: They provided a star-simple drawing of $K_{6}$ that does not contain any \enquote{crossing-free} Hamiltonian cycle (see \Cref{fig:star-simple}). A \emph{star-simple drawing} is a drawing where incident edges must not cross but non-incident edges are allowed to cross more than once. We remark that in the context of star-simple drawings, edges are actually only considered to be \enquote{crossing} when they cross an odd number of times.

\paragraph*{Our contribution.}

Extending the above line of research, we focus on sub-classes of simple drawings and show \Cref{conj:main} to be true for strongly c-monotone drawings (\Cref{cor:cfhc-strong-c-mon}) as well as cylindrical drawings (\Cref{cor:cfhc-cylindrical}). To this end we introduce the following new conjecture.

\begin{conjecture}\label{conj:stronger}
Every simple drawing $\mathcal{D}$ of $K_{n}$ for $n \geq 2$ contains, for each pair of vertices $v_{a}$ and $v_{b}$ in $\mathcal{D}$, a crossing-free Hamiltonian path with end-vertices $v_{a}$ and $v_{b}$.
\end{conjecture}

We show that a positive answer to \Cref{conj:stronger} implies a positive answer to \Cref{conj:main} (\Cref{thm:conj-all-hp-stronger-hc}), which makes \Cref{conj:stronger} the stronger conjecture. We then verify \Cref{conj:stronger} for the classes of strongly c-monotone drawings (\Cref{thm:all-hp-strong-c-mon}) and cylindrical drawings (\Cref{thm:all-hp-cylindrical}). Implicitly, to the best of our knowledge, our results also include the first published proofs of \Cref{conj:main} for $x$-monotone drawings and strongly cylindrical drawings, and cover several more classes of simple drawings.

To better understand the impact of results on sub-classes of simple drawings, we further investigate relations between different such classes (\Cref{fig:the-class-order}). We show that every cylindrical drawing is weakly isomorphic to a c-monotone drawing (\Cref{prop:cylindrical-is-c-mon}) and, especially, that every strongly cylindrical drawing is weakly isomorphic to a strongly \mbox{c-}monotone drawing (\Cref{thm:strong-cylindrical-is-c-mon}). Furthermore, we provide examples of drawings that are strongly c-monotone but not cylindrical and vice versa (\Cref{fig:c-mon-cylindrical-examples}). We give a self-contained proof that $x$-monotone and $x$-bounded drawings of $K_{n}$ are the same up to weak isomorphism (\Cref{thm:x-bound-is-x-mon}) and provide an example showing that this is not true anymore for non-complete graphs (\Cref{fig:x-bound-non-complete}). Finally, we show that there exist $n$-shellable drawings of $K_{n}$ that are not weakly isomorphic to any $x$-bounded drawing (\Cref{obs:n-shell-not-x-mon}), which we do by extending an example from Balko, Fulek, and Kyn{\v{c}}l~\cite{bfk-2015-cnccmd} that intended to prove the same statement.

\subparagraph{Outline.}

The rest of the paper is structured as follows: In \Cref{sec:classes} we formally define several classes of simple drawings, provide background information on them, and put them into relation. \Cref{sec:allpairs} is dedicated to \Cref{conj:stronger}. We first investigate relations between \Cref{conj:stronger,conj:main} in \Cref{sec:paths-imply-cycle}. Then we verify \Cref{conj:stronger} for the above mentioned classes in \Cref{sec:all-pair-proofs}. Before concluding the paper with a discussion in \Cref{sec:conclude}, we study relations between different classes of drawings in \Cref{sec:relations}. In particular, \Cref{sec:x-bound-relate} contains the results on $x$-monotone, $x$-bounded, and $n$-shellable drawings, whereas in \Cref{sec:cylindrical-relate} we show the relations between the different types of cylindrical and c-monotone drawings.

\section{An overview on classes of simple drawings}\label{sec:classes}

In this section we give a detailed overview on classes of simple drawings that have been studied in the literature. Moreover, we define the class of strongly c-monotone drawings and introduce a clear distinction between two types of cylindrical drawings.

The definitions of most of the following classes depend on a certain representation in the plane and a fixed unbounded cell. However, throughout the paper we consider drawings up to weak isomorphism and only use these special representations to simplify the arguments. Especially, all our statements on crossing-free sub-drawings as well as inclusion relations for drawings $\mathcal{D}$ of a certain class $X$ also hold for all drawings $\mathcal{D}'$ that are weakly isomorphic to any drawing in $X$. For the sake of readability we do not spell this out though.

We denote the vertices in a simple drawing by $v_{1}, \ldots, v_{n}$. For an edge $e = \{ v_{a}, v_{b} \}$ we assume by convention that $a < b$. Then, using the total order $<$ on the vertices that is induced by their indices, we relate edges to each other: Given two non-incident edges $e = \{ v_{a}, v_{b} \}$ and $f = \{ v_{c}, v_{d} \}$, without loss of generality $a < c$, we call $e$ and $f$

\begin{itemize}
\item \emph{separated}, if $a < b < c < d$,
\item \emph{linked}\footnote{The setting we call \emph{linked} is also called \emph{crossing} in the literature; see Bar{\'a}t, Gy{\'a}rf{\'a}s, and T{\'o}th~\cite{bgt-2022-mstmocg}. To avoid confusion with two edges actually having a crossing, we chose a different name though.}, if $a < c < b < d$, and
\item \emph{nested}, if $a < c < d < b$.
\end{itemize}
\vspace{-1pt}

\subparagraph{Straight-line drawings.}

Probably the most studied class of simple drawings are \emph{straight-line drawings} in the plane where each edge is the line segment between its two end-vertices. For straight-line drawings of the complete graph $K_{n}$, no three vertices are allowed to lie on a common line, since otherwise edges would pass through vertices. Hence, straight-line drawings of $K_{n}$ are equivalent to $n$ points in general position in the plane. Further, we may assume without loss of generality that no two vertices share the same $x$-coordinate.

The \emph{convex straight-line drawing} on $n$ vertices, which we denote by $\mathcal{C}_{n}$, is the straight-line drawing of $K_{n}$ where all vertices are placed in convex position. We say \enquote{the} drawing because it is unique up to weak isomorphism. A combinatorial description of $\mathcal{C}_{n}$ is as follows: There exists an order of the vertices such that two edges cross if and only if they are linked~\cite{pst-2003-ucctg}.
\vspace{-2pt}

\subparagraph{$X$-monotone and $x$-bounded drawings.}

Let $\mathcal{D}$ be a simple drawing in the plane such that no two vertices have the same $x$-coordinate. If every vertical line in the plane crosses each edge of $\mathcal{D}$ at most once, then we call $\mathcal{D}$ an \emph{$x$-monotone drawing}; see \Cref{fig:x-mon-example} for an example. If each edge is just contained between the vertical lines through its left and right end-vertex, respectively, then we call $\mathcal{D}$ an \emph{$x$-bounded drawing}. $X$-monotone drawings are a natural generalization of straight-line drawings and clearly $x$-bounded drawings are a generalization of $x$-monotone drawings. For both $x$-monotone and $x$-bounded drawings we can assume, up to strong isomorphism, that all vertices lie on a common horizontal line and we label them $v_{1}, \ldots, v_{n}$ from left to right. Also, it can be decided in $\mathcal{O}(n^{5})$ time whether a given simple drawing of $K_{n}$ is weakly isomorphic to an $x$-monotone drawing~\cite{ahpsv-2015-dmgdcg}.
\vspace{-2pt}

\subparagraph{$2$-page-book drawings.}

A further class of simple drawings are \emph{$2$-page-book drawings}. They are defined such that all vertices lie on a common horizontal line, the \emph{spine} of the book, and the relative interior of each edge lies completely above or completely below the spine, which defines the two \emph{pages} of the book; see \Cref{fig:2-page-example} for an example. These drawings are closely related to book embeddings, where more than two pages are allowed; see Bernhart and Kainen~\cite{bk-1979-btg} for details. The following observation can be deduced from the facts that the spine is not crossed by any edge and that two edges in a simple drawing cross at most once.

\begin{observation}\label{obs:2-page-crossings}
Two edges in a $2$-page-book drawing cross if and only if they lie in the same page and are linked with respect to the vertex order along the spine.
\end{observation}

From this it can easily be seen that a simple drawing of $K_{n}$ with $n \geq 3$ is weakly isomorphic to a $2$-page-book drawing if and only if it contains a completely uncrossed\footnote{We use the term \emph{completely uncrossed} for sub-drawings $\mathcal{D}'$ of a simple drawing~$\mathcal{D}$ if no edge of $\mathcal{D}'$ is crossed by any edge of $\mathcal{D}$. With that we aim to clearly distinguish them from \emph{crossing-free} sub-drawings where the edges of $\mathcal{D}'$ are just not crossed by other edges of $\mathcal{D}'$.} Hamiltonian cycle. It also follows that the sub-drawing induced by a single page is weakly isomorphic to a sub-drawing of $\mathcal{C}_{n}$ and that each $2$-page-book drawing is weakly isomorphic to an $x$-monotone drawing, for example, by drawing all edges as half-circles in the respective page. Not all $2$-page-book drawings are weakly isomorphic to a straight-line drawing though. In particular, there are $2$-page-book drawings of $K_{n}$, like the one in \Cref{fig:2-page-example}, with fewer crossings than any straight-line drawing of $K_{n}$ can have~\cite{aafrs-2013-2pcn}.

For the spine of a $2$-page-book drawing we may actually use any simple curve that partitions the plane into two connected components. In particular a circle is possible. In that case we get a circular order $\pi^{\circ}$ on the vertices along the spine. However, we can create a linear order $\pi$ out of $\pi^{\circ}$ by choosing some vertex as $v_{1}$ and some direction along the circle as increasing. If two edges $e$ and $f$ are linked for one possible choice of $\pi$, then $e$ and~$f$ are linked for all possible choices of $\pi$. Hence, \Cref{obs:2-page-crossings} holds for any linear order along a circular spine. In contrast to that, every edge pair that is separated for one such linear order $\pi$ is nested for some other linear order $\pi'$ and vice versa.

\begin{figure}[h]
\centering
\subcaptionbox{\centering\label{fig:x-mon-example}}[.49\textwidth]{\includegraphics[page=17]{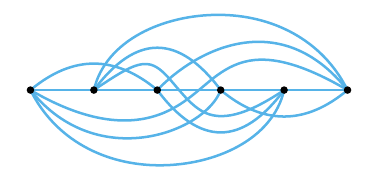}}
\subcaptionbox{\centering\label{fig:2-page-example}}[.49\textwidth]{\includegraphics[page=16]{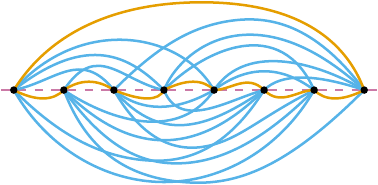}}
\caption{\textbf{(a)}~An $x$-monotone drawing of $K_{6}$. \textbf{(b)}~A crossing minimal $2$-page-book drawing of $K_{8}$. The spine is drawn dashed purple, the completely uncrossed Hamiltonian cycle is drawn orange.}
\label{fig:x-mon-2-page-example}
\vspace{-\baselineskip}
\end{figure}

\subparagraph{Shellable drawings.}

{\'A}brego et al.~\cite{aafrs-2014-sdccn} introduced the notion of \emph{shellable drawings}. A simple drawing $\mathcal{D}$ is \emph{$s$-shellable} if there exists an ordered subset $v_{1}, \ldots, v_{s}$ of its vertices such that the following two properties hold. First, $v_{1}$ and $v_{s}$ are incident to a common cell $F$ of~$\mathcal{D}$. Second, for each $1 \leq i < j \leq s$, both $v_{i}$ and $v_{j}$ are incident to the cell $F_{ij}$ in $\mathcal{D}_{ij}$, where $\mathcal{D}_{ij}$ is the sub-drawing of $\mathcal{D}$ induced by the vertices $v_{i}, \ldots, v_{j}$ and $F_{ij}$ is the cell containing~$F$. It is easy to see that every $x$-bounded drawing of a graph on $n$ vertices is $n$-shellable (short \emph{shellable}) with the vertices ordered in horizontal direction from left to right and $F$ being the unbounded cell. There exist several generalizations of shellable drawings, for example, \emph{bishellable drawings} (see {\'A}brego et al.~\cite{aafmmmrrv-2018-bd}) and \emph{seq-shellable drawings} (see Mutzel and Oettershagen~\cite{mo-2018-cnssdcg}), for which we refer the interested reader to the respective paper.

\subparagraph{C-monotone drawings.}

To characterize the next class, we follow the definition given in~\cite{agtvw-2022-twfpssdcg}. An edge $e$ in a simple drawing is \emph{c-monotone with respect to a point $O$ of the plane} if every ray starting at $O$ crosses $e$ at most once. We call a simple drawing $\mathcal{D}$ in the plane a \emph{c-monotone drawing} if all edges in $\mathcal{D}$ are c-monotone with respect to a common point~$O$ and no two vertices lie on a common ray starting at $O$. The letter \enquote{c} in c-monotone is meant as an abbreviation of \enquote{circularly} and \enquote{$O$} stands for \enquote{the origin}. When we consider a c-monotone drawing $\mathcal{D}$, we assume that a suitable fixed point $O$ is given with it. We can also assume, up to strong isomorphism, that all vertices of $\mathcal{D}$ lie on a common circle with center~$O$.

Apart from being a generalization of $x$-monotone drawings, c-monotone drawings are of interest for the following reason: In a simple drawing $\mathcal{D}$ of $K_{n}$ the star $\mathcal{S}_{v}$ of any vertex $v$ can be extended by a plane spanning tree $\mathcal{T}$ on the $n-1$ vertices of $\mathcal{D}$ excluding $v$ such that also $\mathcal{S}_{v} + \mathcal{T}$ is plane; see Garc{\'i}a, Pilz, and Tejel~\cite[Corollary~3.4]{gpt-2021-psctd}. The structure $\mathcal{S}_{v} + \mathcal{T}$ then either contains large plane sub-drawings itself (like matchings, paths, and cycles) or enforces the existence of a large c-monotone sub-drawing in $\mathcal{D}$. See Fulek~\cite[Proof of Theorem~1.1]{f-2014-endestgcd} for more details on the idea behind this.

\begin{figure}[h]
\centering
\subcaptionbox{\centering\label{fig:strong-c-mon_a}}[.328\textwidth]{\includegraphics[page=3]{Figures/special_c_mon.pdf}}
\subcaptionbox{\centering\label{fig:strong-c-mon_b}}[.328\textwidth]{\includegraphics[page=2]{Figures/special_c_mon.pdf}}
\subcaptionbox{\centering\label{fig:strong-c-mon_c}}[.328\textwidth]{\includegraphics[page=1]{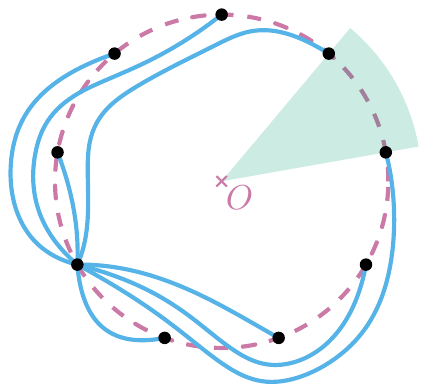}}
\caption{\textbf{(a)}~The wedge $\Lambda_{e}$ (shaded seagreen) of the lightblue edge $e$. The darkblue edge $f$ together with $e$ covers the plane, which is not allowed in a strongly c-monotone drawing. \textbf{(b)}~Illustration of \Cref{lem:strong-c-mon-for-kn}: A pair of non-incident edges (blue) covering the plane enforces a pair of incident edges covering the plane; either with the darker ($g_{1}$) or with the lighter ($g_{2}$) version of $g$. \textbf{(c)}~Condition 3 of \Cref{lem:strong-c-mon-for-kn}: The star is not crossed by any ray within the seagreen wedge.}
\label{fig:strong-c-mon}
\vspace{-\baselineskip}
\end{figure}

\subparagraph{Strongly c-monotone drawings.}

Given an edge $e$ in a c-monotone drawing, we define the \emph{wedge $\Lambda_{e}$ of~$e$} to be the closed wedge with apex~$O$ that has the rays from $O$ through the end-vertices of $e$ as its boundary and contains $e$. For a set of edges $E$ we say that the edges in $E$ \emph{cover the plane} if the union of all their wedges is the whole plane. See \Cref{fig:strong-c-mon_a} for an illustration of those terms. Clearly it needs at least two edges to cover the plane. This leads us to the definition of strongly c-monotone drawings. If no pair of edges in a c-monotone drawing $\mathcal{D}$ covers the plane, then we call $\mathcal{D}$ \emph{strongly c-monotone}. Equivalently, a c-monotone drawing is strongly c-monotone if, for each pair of edges $e$ and $f$ in $\mathcal{D}$, there exists a ray starting at~$O$ that crosses neither $e$ nor $f$. Note that all $x$-monotone drawings are also strongly c-monotone: Choose $O$ sufficiently far below the drawing such that every ray starting at $O$ crosses each edge of $\mathcal{D}$ at most once; then every ray pointing away from $\mathcal{D}$ does not cross any edge of $\mathcal{D}$. For $K_{n}$, strong c-monotonicity can be characterized as follows.

\begin{lemma}\label{lem:strong-c-mon-for-kn}
Let $\mathcal{D}$ be a c-monotone drawing of $K_{n}$. Then the following are equivalent:
\begin{enumerate}
\item No pair of edges $e$ and $f$ in $\mathcal{D}$ covers the plane.
\item No pair of incident edges $e$ and $f$ in $\mathcal{D}$ covers the plane.
\item No star in $\mathcal{D}$ covers the plane.
\end{enumerate}
\end{lemma}

\begin{proof}
Clearly, the first condition implies the second. For the second implying the first, assume that two non-incident edges $e = \{ v_{a}, v_{b} \}$ and $f = \{ v_{c}, v_{d} \}$ cover the plane. Then the wedges $\Lambda_{e}$ of $e$ and $\Lambda_{f}$ of $f$ intersect in two wedges with apex~$O$. Without loss of generality, they do so between $v_{a}$ and $v_{c}$, and between $v_{b}$ and $v_{d}$; see \Cref{fig:strong-c-mon_b} for an illustration. We now consider the edge $g = \{ v_{a}, v_{d} \}$. Either $g$ is contained in~$\Lambda_{e}$, then the incident edges $g$ and $f$ cover the plane. Or $g$ is contained in $\Lambda_{f}$ and the incident edges $g$ and $e$ cover the plane. This establishes equivalence of the first two conditions. The equivalence of conditions two and three is straightforward: A star covers the plane if and only if two of its edges, which are incident edges, cover the plane.
\end{proof}

Especially the third condition is quite easy to verify, at least by hand, only requiring $n$ checks instead of the $\mathcal{O}(n^{4})$ checks when considering all edge pairs; see also \Cref{fig:strong-c-mon_c}.

\subparagraph{Cylindrical drawings.}

Following the definition introduced in~\cite{aafrs-2014-sdccn}, we call a simple drawing \emph{cylindrical} if all vertices lie on two concentric circles and no edge crosses any of these two circles. The name \enquote{cylindrical} comes from picturing it as being drawn on the surface of a cylinder: The vertices are placed on the two rims and each edge lies completely in the top lid, the lateral face, or the bottom lid. This is a generalization of Hill's conjectured crossing minimal drawing of $K_{n}$, which we denote by $\mathcal{Z}_{n}$. It can be achieved by placing half of the vertices equally spaced on each rim and drawing the edges as geodesics. See \Cref{fig:hill} for a depiction and Harary and Hill~\cite{hh-1963-nccg} or Guy, Jenkyns, and Schaer~\cite{gjs-1968-tcncg} for more details.

For cylindrical drawings in the plane, as defined above, we call the area outside the outer circle, between the two circles, and inside the inner circle \emph{outer}, \emph{lateral}, and \emph{inner face}, respectively. Additionally, we call edges connecting two vertices from different circles \emph{lateral edges} and edges connecting two vertices on the inner or outer circle \emph{inner} or \emph{outer}, respectively, \emph{circle edges}. In particular, we call circle edges connecting two neighboring vertices on their circle \emph{rim edges}. See \Cref{fig:hill-cylindrical} for some illustrations of these terms.

Because the two concentric circles are not allowed to be crossed by any edge, all lateral edges have to lie in the lateral face. In contrast to that, however, inner (outer, respectively) circle edges can lie either in the inner (outer, respectively) face or in the lateral face. Also, the following observation is implied by the definition.

\begin{observation}\label{obs:cylindrical-2-page}
Let $\mathcal{D}$ be a cylindrical drawing. Then the sub-drawing of $\mathcal{D}$ induced by all vertices on one of the circles is strongly isomorphic to a $2$-page-book drawing.
\end{observation}

\begin{figure}[h]
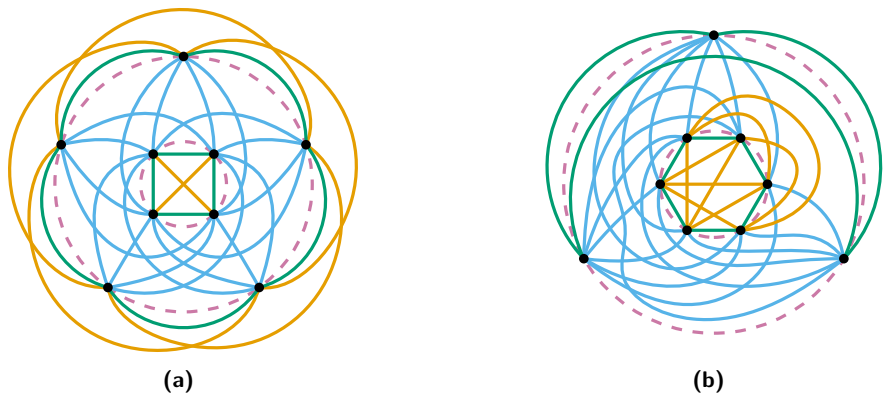

\centering
\subcaptionbox{\centering\label{fig:hill}}[.49\textwidth]{\includegraphics[page=11]{Figures/special_c_mon.pdf}}
\subcaptionbox{\centering\label{fig:cylindrical}}[.49\textwidth]{\includegraphics[page=12]{Figures/special_c_mon.pdf}}
\caption{\textbf{(a)}~Hill's drawing $\mathcal{Z}_{9}$ drawn in the plane. The two concentric circles are drawn violet, the rim edges are seagreen, all other circle edges are orange, and the lateral edges are lightblue. \textbf{(b)}~An arbitrary (not strongly) cylindrical drawing of $K_{9}$ with the same color coding.}
\label{fig:hill-cylindrical}
\vspace{-\baselineskip}
\end{figure}

\subparagraph{Strongly cylindrical drawings.}

If in a cylindrical drawing all inner circle edges lie in the inner face and all outer circle edges lie in the outer face, then we call the drawing \emph{strongly cylindrical}. Especially, Hill's drawing $\mathcal{Z}_{n}$ is strongly cylindrical. In the literature the term \enquote{cylindrical drawing} is also used for what we call strongly cylindrical drawing. Further, Ruiz-Vargas~\cite{r-2017-mdetg} defines arbitrary drawings on the surface of an infinite open cylinder as \enquote{cylindrical drawings} and his \enquote{monotone cylindrical drawings} are equivalent to our c-monotone drawings.

\begin{figure}[h]
\centering
\subcaptionbox{\centering\label{fig:G5}}[.194\textwidth]{\includegraphics[page=2]{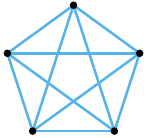}}
\subcaptionbox{\centering\label{fig:G3}}[.194\textwidth]{\includegraphics[page=3]{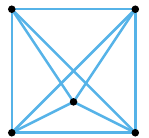}}
\subcaptionbox{\centering\label{fig:G1}}[.194\textwidth]{\includegraphics[page=4]{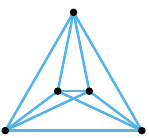}}
\subcaptionbox{\centering\label{fig:T3}}[.194\textwidth]{\includegraphics[page=5]{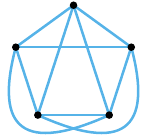}}
\subcaptionbox{\centering\label{fig:T5}}[.194\textwidth]{\includegraphics[page=6]{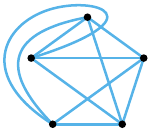}}
\caption{The five non-isomorphic simple drawings of $K_{5}$: \textbf{(a)}~The convex straight-line drawing~$\mathcal{C}_{5}$. \textbf{(a) -- (c)}~The three types of straight-line drawings of $K_{5}$. \textbf{(c)}~In some sense Hill's drawing~$\mathcal{Z}_{5}$. \textbf{(e)}~The twisted drawing~$\mathcal{T}_{5}$. \textbf{(a)~and~(e)}~The two crossing maximal drawings of $K_{5}$.}
\label{fig:five-K5s}
\vspace{-\baselineskip}
\end{figure}

\subparagraph{Crossing maximal drawings.}

In a simple drawing $\mathcal{D}$ of $K_{n}$ any sub-drawing induced by~$4$ vertices can admit at most one crossing. If every $4$-tuple of vertices induces a crossing and hence  $\mathcal{D}$ contains exactly $\binom{n}{4}$ crossings, then we call $\mathcal{D}$ \emph{crossing maximal}. For example, the drawings of $K_{5}$ in \Cref{fig:G5,fig:T5}, and in general $\mathcal{C}_{n}$, are crossing maximal.

\begin{figure}[b]
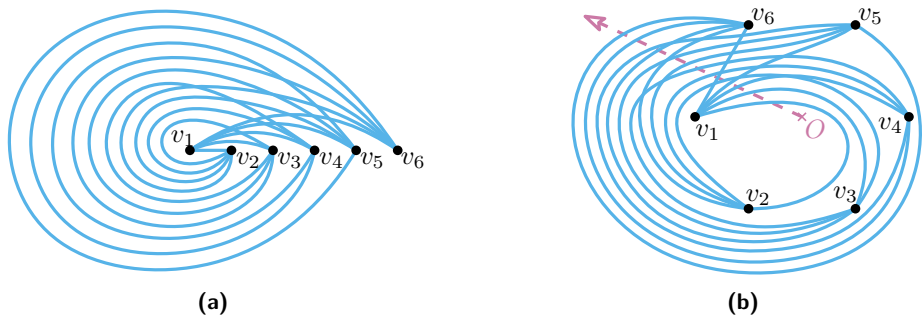

\centering
\subcaptionbox{\centering\label{fig:twisted_a}}[.49\textwidth]{\includegraphics[page=9]{Figures/special_c_mon.pdf}}
\subcaptionbox{\centering\label{fig:twisted_b}}[.49\textwidth]{\includegraphics[page=10]{Figures/special_c_mon.pdf}}
\caption{The twisted drawing $\mathcal{T}_{6}$ drawn \textbf{(a)}~in the usual way and \textbf{(b)}~as a c-monotone drawing.}
\label{fig:twisted}
\end{figure}

\subparagraph{Twisted and generalized twisted drawings.}

In their study of crossing maximal drawings, Harborth and Mengersen~\cite{hm-1992-dcgmnc} introduced another special drawing of~$K_{n}$, which generalizes the drawing in \Cref{fig:T5}. Harborth~\cite{h-1998-etdcg} further analyzed that drawing in the context of empty triangles and Pach, Solymosi, and T{\'o}th~\cite{pst-2003-ucctg} later named it twisted. A simple drawing of $K_{n}$ is a \emph{twisted drawing} if there exists an order on its vertices such that two edges cross if and only if they are nested. This uniquely defines the twisted drawing of~$K_{n}$, which we denote by~$\mathcal{T}_{n}$, up to weak isomorphism and we use this defining order to label its vertices from $v_{1}$ to~$v_{n}$. One possible way of realizing this special crossing property is shown in \Cref{fig:twisted_a}. Also, the drawing in \Cref{fig:x-mon-example} is actually an $x$-monotone representation of~$\mathcal{T}_{6}$; for $n \geq 7$ this is not possible anymore. However, $\mathcal{T}_{n}$ can be drawn c-monotone such that there exists a ray starting at $O$ that crosses all edges; \Cref{fig:twisted_b} shows an example. This is the defining property for the following generalization of twisted drawings~\cite{agtvw-2022-twfpssdcg}: A c-monotone drawing~$\mathcal{D}$ of $K_{n}$ is called \emph{generalized twisted} (short \emph{g-twisted}) if there exists a ray starting at $O$ that crosses all edges of $\mathcal{D}$. It has been shown~\cite{agtvw-2023-crsgtd5t} that a simple drawing $\mathcal{D}$ of $K_{n}$, for $n \geq 7$, is weakly isomorphic to a g-twisted drawing if and only if every sub-drawing of~$\mathcal{D}$ induced by $5$ vertices is weakly isomorphic to the twisted drawing $\mathcal{T}_{5}$. This especially implies that all g-twisted drawings are crossing maximal drawings. For more information on g-twisted drawings we refer the interested reader to~\cite{agtvw-2022-twfpssdcg,agtvw-2023-crsgtd5t,gtvw-2023-etgtd}.

\vspace{-1pt}
\subparagraph{Generalized convex drawings.}

Arroyo, McQuillan, Richter, and Salazar~\cite{amrs-2021-cdcgtg} introduced the following class of simple drawings: Every triangle in a simple drawing separates the plane into two connected components, the closures of which are called the \emph{sides} of the triangle. A side $S$ of a triangle is called \emph{convex} if, for each pair of vertices $v_{a}$ and $v_{b}$ in~$S$, the edge $\{ v_{a}, v_{b} \}$ is completely contained in~$S$. We call a simple drawing $\mathcal{D}$ of $K_{n}$ \emph{generalized convex}\footnote{Arroyo et al.~\cite{amrs-2018-llpdet,amrs-2021-cdcgtg,ars-2021-edcgap} called them just \enquote{convex drawings}. But because in these drawings both sides of a triangle can be convex we prefer to call them \enquote{generalized convex drawings}.} (short \emph{g-convex}) if every triangle in $\mathcal{D}$ has a convex side. It has been shown~\cite[Theorem~2.6]{amrs-2021-cdcgtg} that a drawing is g-convex if and only if every sub-drawing induced by $5$ vertices is weakly isomorphic to one of the three types of straight-line drawings of $K_{5}$; see \Cref{fig:G5,fig:G3,fig:G1}.

Further, there exists the sub-class of \emph{hereditarily convex drawings} (short \emph{h-convex}) where the convex side of triangles is inherited through inclusion. Arroyo, Richter, and Sunohara~\cite{ars-2021-edcgap} showed that, for $K_{n}$, this is equivalent to so-called \emph{pseudospherical drawings}, a generalization of drawings on the sphere with the edges drawn as geodesics. Going down one more step, an h-convex drawing $\mathcal{D}$ is called \emph{face convex} (short \emph{f-convex}) if there exists a cell $F$ in $\mathcal{D}$ such that for every triangle in $\mathcal{D}$ the side not containing $F$ is convex. For $K_{n}$, Arroyo, McQuillan, Richter, and Salazar~\cite{amrs-2018-llpdet} proved this to be equivalent to the well-known \emph{pseudolinear drawings}. And those are in turn a generalization of straight-line drawings and a sub-class of $x$-monotone drawings. For more information on g-convex drawings and their sub-classes see~\cite{amrs-2018-llpdet,amrs-2021-cdcgtg,ars-2021-edcgap}.

\subparagraph{Overview.}

We finish this section with an overview on all the introduced drawing classes and an outlook on what we will show in the following sections. Since a picture is worth a thousand words, we do so with \Cref{fig:the-class-order}.

\begin{figure}[h]
\centering
\includegraphics[page=1]{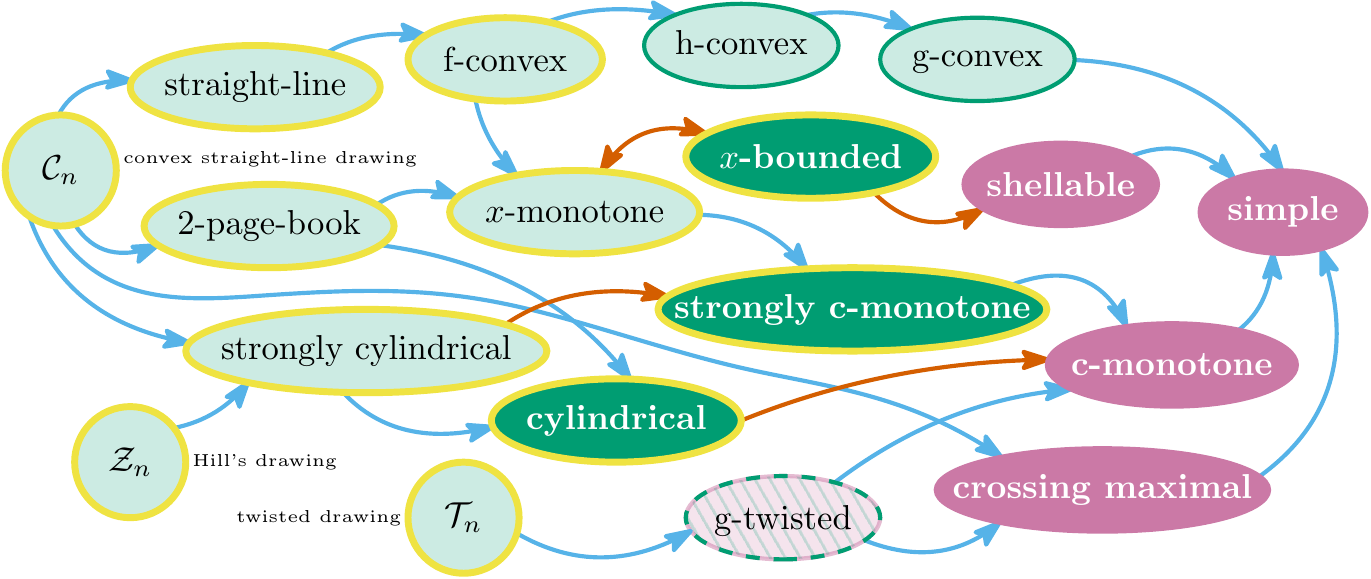}
\caption{An overview on different classes of simple drawings of $K_{n}$. Single arrows indicate that the source class is a proper subset of the target class, regarding weak isomorphism; a double arrow means equality. Each darkorange arrow indicates a result shown in this work. The seagreen disks mark classes for which \Cref{conj:main} is (now) proven; we focus on the darker seagreen ones. All classes for which we verify \Cref{conj:stronger} are emphasized with a yellow boundary. For violet classes both conjectures are still open, while for g-twisted drawings \Cref{conj:main} is partially proven.}
\label{fig:the-class-order}
\vspace{-0.8\baselineskip}
\end{figure}

\section{Crossing-free Hamiltonian paths between all vertex pairs}\label{sec:allpairs}

This section is devoted to showing that \Cref{conj:stronger} is indeed a strengthening of \Cref{conj:main} and to proving these conjectures for strongly c-monotone drawings and cylindrical drawings. Since the question regarding a crossing-free Hamiltonian path between each pair of vertices is new (\Cref{conj:stronger}), it is prudent to check whether this actually makes sense. It~can be shown that in every straight-line drawing of $K_{n}$ all those crossing-free Hamiltonian paths actually exist. The main idea for this is to visit all vertices in circular order around some start-vertex. We leave a formal proof to the reader because the statement also follows from more general results later in this paper.

As another sanity check, using our rotation system database~\cite{aafhpprsv-2015-agdscg}, we can confirm \Cref{conj:stronger} for all simple drawings on $n \leq 9$ vertices. Moreover, Manfred Scheucher (personal communication) used the SAT framework from~\cite{bfrs-2023-usspssd} to confirm \Cref{conj:stronger} for all simple drawings on $n \leq 10$ vertices and for all g-convex drawings on $n \leq 12$ vertices.

A problem related to prespecifying both end-vertices of a crossing-free Hamiltonian path~$\mathcal{P}$ would be to fix a certain edge to be part of $\mathcal{P}$. In straight-line drawings of $K_{n}$ this is easily possible because the supporting line of an edge splits the vertex set into two parts with disjoint convex hulls. Moreover, in Theorem 4.1 of the arXiv version of~\cite{bfrs-2023-usspssd}, Bergold, Felsner, M.\ Reddy, and Scheucher showed that even in g-convex drawings $\mathcal{D}$ of $K_{n}$ there exists, for each edge $e$ in $\mathcal{D}$, a crossing-free Hamiltonian path containing $e$. However, in an arbitrary simple drawing of $K_{n}$ this is not possible. For example, choose the edge $\{ v_{1} , v_{n} \}$ in the twisted drawing $\mathcal{T}_{n}$: Since it crosses all non-incident edges, it cannot be part of any crossing-free path with more than three edges.

\subsection{All paths conjecture implies cycle conjecture}\label{sec:paths-imply-cycle}

As a first main result of this paper we establish a connection between \Cref{conj:main,conj:stronger}. We start by observing that every completely uncrossed edge in a simple drawing $\mathcal{D}$ can be added to a crossing-free sub-drawing of $\mathcal{D}$ without introducing any new crossings.

\begin{observation}\label{obs:conjectures-uncrossed-edge}
Let $\mathcal{D}$ be a simple drawing of $K_{n}$ with $n \geq 3$ that contains a completely uncrossed edge $e = \{ v_{a}, v_{b} \}$ and a crossing-free Hamiltonian path with end-vertices $v_{a}$ and $v_{b}$. Then $\mathcal{D}$ contains a crossing-free Hamiltonian cycle.
\end{observation}

This observation establishes that \Cref{conj:stronger} is at least as strong as \Cref{conj:main} for all drawings containing a completely uncrossed edge. This includes, for example, $x$-bounded drawings because with the vertices ordered from left to right, two separated edges can never cross in an $x$-bounded drawing. Especially, we have the following.

\begin{observation}\label{obs:x-bound-uncrossed}
Let $\mathcal{D}$ be an $x$-bounded drawing of $K_{n}$ with vertices $v_{1}, \ldots, v_{n}$ labeled from left to right. Then the edges $\{ v_{1}, v_{2} \}$ and $\{ v_{n-1}, v_{n} \}$ are completely uncrossed.
\end{observation}

We next show that also every cylindrical drawing contains completely uncrossed edges. For a circle edge $e = \{ v_{a}, v_{b} \}$ in the lateral face and end-vertices on circle~$C$, we consider the unique area $F$ that is completely contained in the lateral face, and only bounded by $e$ and a part of $C$; see \Cref{fig:cylindrical-crossings} for an example illustration. We say that $e$ \emph{guards a vertex} $v$ on $C$ if $v$ lies on the boundary of $F$, that is, $e$ guards all vertices between and including $v_{a}$ and $v_{b}$ on $C$, either in clockwise or in counter-clockwise direction. In particular, $e$ always guards its own end-vertices. Similar to \Cref{obs:2-page-crossings} we get the following.

\begin{observation}\label{obs:cylindrical-crossings}
Let $e$ be a circle edge in the lateral face of a cylindrical drawing $\mathcal{D}$. Let $f$ be another edge of $\mathcal{D}$ that is non-incident to $e$ and also lies in the lateral face of $\mathcal{D}$. Then $e$ and $f$ cross if and only if $e$ guards exactly one end-vertex of $f$.
\end{observation}

We further say that $e$ \emph{guards another edge} $f$ if $f$ also lies in the lateral face and both its end-vertices are guarded by $e$, which is equivalent to $f$ lying completely in $F$. It therefore follows that two edges cannot mutually guard one another.

\begin{observation}\label{obs:guarding-edges}
Let $e$ be a circle edge in the lateral face of a cylindrical drawing $\mathcal{D}$, let $V_{e}$ be the set of vertices guarded by $e$, and let $e$ guard another circle edge $f$. Then the set of vertices guarded by $f$ is a proper subset of $V_{e}$.
\end{observation}

\begin{figure}[h]
\centering
\includegraphics[page=13]{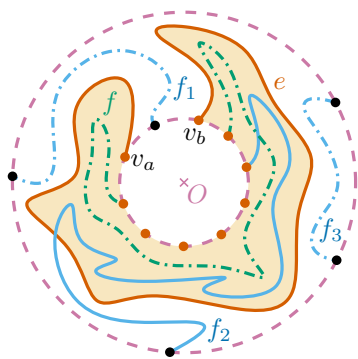}
\caption{The orange circle edge $e$ guards all orange vertices and the seagreen edge $f$. The area $F$ is shaded lightorange. In support of \Cref{obs:cylindrical-crossings}, three more examples for edges $f_{i}$ are drawn in lightblue. The solid edge $f_{2}$ crosses $e$, whereas all dash dotted edges do not cross $e$.}
\label{fig:cylindrical-crossings}
\end{figure}

With these observations we can now show the following fundamental structural statement for cylindrical drawings, which we also use later for most of the results on these drawings.

\begin{lemma}\label{lem:cylindrical-rim}
Let $\mathcal{D}$ be a cylindrical drawing. Then all but at most two rim edges in $\mathcal{D}$, at most one per circle, are completely uncrossed.
\end{lemma}

\begin{proof}
Consider, without loss of generality, the inner circle. First, if a rim edge $e$ lies in the inner face, then $e$ is uncrossed by \Cref{obs:2-page-crossings,obs:cylindrical-2-page}. Further, if $e$ lies in the lateral face, then there are two more cases. Either $e$ only guards its own end-vertices, then it is uncrossed by \Cref{obs:cylindrical-crossings}. Or $e$ guards all vertices on the inner circle. Then $e$ can have crossings but $e$ also guards all other inner rim edges that lie in the lateral face. Especially, by \Cref{obs:guarding-edges}, all those rim edges guarded by $e$ only guard their own end-vertices. Hence, by the first two cases, all other inner rim edges are uncrossed.
\end{proof}

\begin{corollary}\label{cor:cylindrical-conjectures}
For every cylindrical or $x$-bounded drawing, as well as every other drawing containing a completely uncrossed edge, \Cref{conj:stronger} implies \Cref{conj:main}.
\end{corollary}

Unfortunately, not all simple drawings of $K_{n}$ have completely uncrossed edges. Harborth and Mengersen~\cite{hm-1974-ewcdcg} already observed that for every $n \geq 8$ there exist simple drawings of $K_{n}$ where every edge is crossed. Moreover, Kyn{\v{c}}l and Valtr~\cite{kv-2009-ecfoestcg} showed that there are simple drawings of $K_{n}$ where every edge is crossed by as many as $\Omega(n^{3/2})$ other edges. However, we can still prove a relation between \Cref{conj:main,conj:stronger} for those drawings, by considering all simple drawings of $K_{n}$ at once.

\begin{theorem}\label{thm:conj-all-hp-stronger-hc}
Let $n \geq 3$ be fixed. If \Cref{conj:stronger} is true for all simple drawings of~$K_{n+1}$, then \Cref{conj:main} is true for all simple drawings of $K_{n}$.
\end{theorem}

\begin{proof}
Let $\mathcal{D}$ be an arbitrary simple drawing of $K_{n}$ for some fixed $n \geq 3$ and assume \Cref{conj:stronger} is true for all simple drawings of $K_{n+1}$. Consider the vertex $v_{n} \in \mathcal{D}$ and, especially, the star of $v_{n}$. Let, without loss of generality, the rotation of $v_{n}$ be $v_{1}, \ldots, v_{n-1}$ in clockwise order. We produce a simple drawing $\mathcal{D}'$ of $K_{n+1}$ by duplicating $v_{n}$ in the following way (see \Cref{fig:conjectures-reduction} for an illustration of the construction). We place a new vertex $v_{n+1}$ close to $v_{n}$ into the cell of $\mathcal{D}$ that is incident to $v_{n}$ and has the edges $\{ v_{n}, v_{n-1} \}$ and $\{ v_{n}, v_{1} \}$ on its boundary. We then add, for $1 \leq i \leq n-1$, the edges $\{ v_{n+1}, v_{i} \}$ by starting at $v_{n+1}$, going around $v_{n}$ in counter-clockwise direction until we reach the edge $\{ v_{n}, v_{i} \}$, and then following $\{ v_{n}, v_{i} \}$ until we reach $v_{i}$. That way, the edge $\{ v_{n+1}, v_{i} \}$ crosses exactly all edges $\{ v_{n}, v_{j} \}$, for $i < j < n$, and all edges that are crossed by $\{ v_{n}, v_{i} \}$ in $\mathcal{D}$. Finally, we connect $v_{n}$ and $v_{n+1}$ by a completely uncrossed edge.

By assumption, $\mathcal{D}'$ contains a crossing-free Hamiltonian path $\mathcal{P}$ with end-vertices $v_{n}$ and~$v_{n+1}$. Because $n+1 \geq 4$, the path $\mathcal{P}$ uses neither both of $\{ v_{i}, v_{n} \}$ and $\{ v_{i}, v_{n+1} \}$ for any vertex $v_{i}$ in $\mathcal{D}'$ with $1 \leq i \leq n-1$ nor the edge $\{ v_{n}, v_{n+1} \}$. Therefore, identifying $v_{n}$ and $v_{n+1}$ in $\mathcal{P}$ produces a Hamiltonian cycle $\mathcal{C}$ in $\mathcal{D}$. Since all edges $\{ v_{i}, v_{n+1} \}$ have the same crossings as $\{ v_{i}, v_{n} \}$, except for crossings with edges incident to $v_{n}$ and $v_{n+1}$, the cycle $\mathcal{C}$ is also crossing-free. And as $\mathcal{D}$ was arbitrary, this finishes the proof.
\end{proof}

\begin{figure}[h]
\centering
\subcaptionbox{\centering\label{fig:conjectures-reduction}}[.62\textwidth]{\includegraphics[page=7]{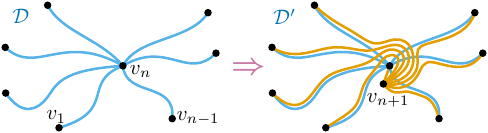}}
\subcaptionbox{\centering\label{fig:reduction-c-mon}}[.36\textwidth]{\includegraphics[page=8]{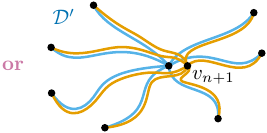}}
\caption{\textbf{(a)}~The basic construction to add a vertex to a simple drawings of $K_{n}$. \textbf{(b)}~A possible different placement of $v_{n+1}$ together with a different routing of its incident edges.}
\label{fig:add-a-vertex}
\end{figure}

We remark that the described construction was also used by Harborth and Mengersen~\cite{hm-1992-dcgmnc} to inductively produce crossing maximal drawings of $K_{n}$. Further, if we want an analogous statement of \Cref{thm:conj-all-hp-stronger-hc} for only a sub-class of simple drawings, then we have to make sure that the drawing $\mathcal{D}'$ can be constructed in such a way that it also lies in that sub-class. For that, we actually have several possibilities to vary the construction without changing the essential crossing properties of it. First, the vertex $v_{n+1}$ can be placed in any of the $n-1$ cells incident to $v_{n}$. Second, there are $n$ different ways of how to split the edges incident to $v_{n+1}$ into two sets, regarding whether they go around $v_{n}$ in clockwise or counter-clockwise direction. See \Cref{fig:reduction-c-mon} for an example. In particular, for strongly c-monotone drawings, we can place $v_{n+1}$ directly next to $v_{n}$ in circular order around $O$ and draw all edges such that they are c-monotone with respect to $O$. It then follows that also $\mathcal{D}'$ is strongly c-monotone (by \Cref{lem:strong-c-mon-for-kn}, we only need to check each star in $\mathcal{D}'$).

\begin{corollary}\label{cor:strong-c-mon-conjectures}
If \Cref{conj:stronger} is true for all strongly c-monotone drawings, then also \Cref{conj:main} is true for all strongly c-monotone drawings.
\end{corollary}

We note that, for example, for the classes of g-twisted or g-convex drawings, it is not immediately clear whether an analogous statement holds.

\subsection{Proving the all paths conjecture for sub-classes of simple drawings}\label{sec:all-pair-proofs}

In this section we prove \Cref{conj:stronger} for strongly c-monotone drawings and for cylindrical drawings. As a basis for both cases, we first verify the conjecture for $x$-monotone drawings. We formulate the following definitions for $x$-bounded drawings because we will need them later in this more general setting. Given an edge $e = \{ v_{a}, v_{b} \}$ in an $x$-bounded drawing $\mathcal{D}$, we consider the area $F$ between and including the vertical lines through $v_{a}$ and $v_{b}$, respectively. The edge $e$ splits $F$ into two connected components, the closures thereof we call the two \emph{sides} of $e$. If a sub-drawing of $\mathcal{D}$, in particular an edge or a vertex, is completely contained in one side of $e$, we say that it lies \emph{above} or \emph{below} $e$. The following observation can be deduced from the definition of $x$-bounded drawings and the fact that two edges in a simple drawing have at most one point in common.

\begin{observation}\label{obs:x-bound-edge-above-below}
Let $e$ and $f$ be two edges of an $x$-bounded drawing $\mathcal{D}$ such that both end-vertices of $f$ lie in the same side of $e$. Then the whole edge $f$ lies in that side.
\end{observation}

With this, we can inductively combine crossing-free paths on certain sub-drawings of $\mathcal{D}$ to a crossing-free Hamiltonian path in $\mathcal{D}$. For convenience we state the following proposition only for $x$-monotone drawings. An analogous statement for $x$-bounded drawings will later follow from \Cref{thm:x-bound-is-x-mon}.

\begin{proposition}\label{prop:all-hp-x-mon}
Let $\mathcal{D}$ be an $x$-monotone drawing of $K_{n}$ and let $v_{a}$ and $v_{b}$ be two vertices in $\mathcal{D}$. Then $\mathcal{D}$ contains a crossing-free Hamiltonian path with end-vertices $v_{a}$ and $v_{b}$.
\end{proposition}

\begin{proof}
We prove the statement by induction on $n$. For $n \leq 3$, it is obviously true. For $n > 3$, let the vertices $v_{1} , \ldots , v_{n}$, as usual, be labeled from left to right in horizontal direction.

Let first $v_{a}$ and $v_{b}$, with $a, b \not\in \{ 1, n \}$, lie in different sides of the edge $e = \{ v_{1}, v_{n} \}$, without loss of generality, $v_{a}$ above and $v_{b}$ below. See \Cref{fig:x-mon-all-pairs_a} for an illustration of the following. The sub-drawing induced by all vertices above $e$ excluding $v_{n}$ and the sub-drawing induced by all vertices below $e$ excluding $v_{1}$ are proper sub-drawings of $\mathcal{D}$ and clearly $x$-monotone. Hence, by the induction hypothesis, there exists a crossing-free path $\mathcal{P}_{1}$ with end-vertices $v_{a}$ and $v_{1}$, visiting all vertices above $e$ except $v_{n}$, and another crossing-free path $\mathcal{P}_{2}$ with end-vertices $v_{n}$ and $v_{b}$, visiting all vertices below $e$ except $v_{1}$. Further, by \Cref{obs:x-bound-edge-above-below}, $e$ separates $\mathcal{P}_{1}$ and $\mathcal{P}_{2}$. Therefore, combining $\mathcal{P}_{1}$ and $\mathcal{P}_{2}$ via $e$ creates a crossing-free Hamiltonian path with end-vertices $v_{a}$ and $v_{b}$.

Let next $v_{a}$ and $v_{b}$, with $1 < a < b$, lie in the same side of $e$, without loss of generality, above $e$. See \Cref{fig:x-mon-all-pairs_b} for an illustration. Then, by the induction hypothesis and similar to before, there exists a crossing-free path $\mathcal{P}_{1}$ with end-vertices $v_{a}$ and $v_{1}$, visiting all vertices above $e$ that lie to the left of $v_{b}$, another crossing-free path $\mathcal{P}_{2}$ with end-vertices $v_{1}$ and~$v_{n}$, visiting all vertices below $e$, and a third crossing-free path $\mathcal{P}_{3}$ with end-vertices $v_{n}$ and~$v_{b}$, visiting all vertices between $v_{n}$ and $v_{b}$ that lie above~$e$. The path $\mathcal{P}_{3}$ may just be a single vertex, which happens if $b = n$, and at most $\mathcal{P}_{2}$ can contain $e$, which happens if there are no vertices strictly below $e$. Further, by $x$-monotonicity, $\mathcal{P}_{1}$ cannot cross $\mathcal{P}_{3}$ and, by \Cref{obs:x-bound-edge-above-below}, $\mathcal{P}_{2}$ cannot cross any of the other two paths. Therefore, combining all three paths creates a crossing-free Hamiltonian path with end-vertices $v_{a}$ and~$v_{b}$. The case $a < b < n$ works similarly with the arguments being right-left reversed.

Finally, if $v_{a} = v_{1}$ and $v_{b} = v_{n}$, then the path in given order from left to right is a crossing-free Hamiltonian path with end-vertices $v_{a}$ and $v_{b}$.
\end{proof}

We next show that strongly c-monotone drawings behave locally very much like $x$-mono\-tone drawings. Recall the definition of the wedge $\Lambda_{e}$ of an edge $e$ from \Cref{sec:classes} and be aware that the following lemma does in general not hold for (non-strongly) c-monotone drawings.

\begin{figure}[t]
\centering
\subcaptionbox{\centering\label{fig:x-mon-all-pairs_a}}[.49\textwidth]{\includegraphics[page=1]{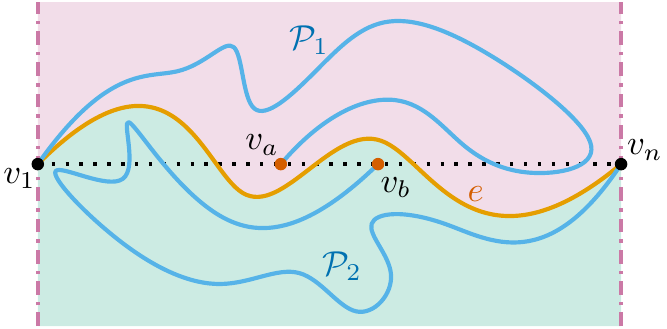}}
\subcaptionbox{\centering\label{fig:x-mon-all-pairs_b}}[.49\textwidth]{\includegraphics[page=2]{Figures/existence_all_pairs.pdf}}
\caption{The two cases for finding a crossing-free Hamiltonian path between two given vertices $v_{a}$ and $v_{b}$ in an $x$-monotone drawing of $K_{n}$ when $v_{a}$ and $v_{b}$ lie \textbf{(a)}~in different sides or \textbf{(b)}~in the same side of the edge $e = \{ v_{1}, v_{n} \}$. Edges are drawn orange while longer paths are lightblue.}
\label{fig:x-mon-all-pairs}
\end{figure}

\begin{lemma}\label{lem:strong-c-mon-x-mon}
Let $e$ be an edge in a strongly c-monotone drawing $\mathcal{D}$. Then the sub-drawing $\mathcal{D}_{e}$ of $\mathcal{D}$, induced by all vertices in the wedge $\Lambda_{e}$ of $e$, is fully contained in $\Lambda_{e}$ and strongly isomorphic to an $x$-monotone drawing.
\end{lemma}

\begin{proof}
Suppose that an edge $f$ with both end-vertices in $\Lambda_{e}$ leaves $\Lambda_{e}$ through one of its boundary rays. Then $f$ has to encircle $O$ and enter $\Lambda_{e}$ again through its other boundary ray. Consequently $e$ and $f$ cover the plane, a contradiction to strong c-monotonicity. Hence, $\mathcal{D}_{e}$~is fully contained in $\Lambda_{e}$. Furthermore, there exists a homeomorphism of the plane that maps $\mathcal{D}_{e}$ to an $x$-monotone drawing, for example, by first rotating the plane accordingly around $O$ and then moving $O$ to infinity through the complement of $\Lambda_{e}$.
\end{proof}

For a c-monotone drawing $\mathcal{D}$ with origin $O$ and two consecutive vertices in the circular order around $O$, we denote the closed wedge with apex $O$ that has the rays from $O$ through these two vertices as its boundary and does not contain any other vertex of $\mathcal{D}$ as a \emph{gap} of $\mathcal{D}$. Further, we call an edge between two consecutive vertices a \emph{gap edge} and associate it with the respective gap. Then the wedge $\Lambda_{e}$ of a gap edge $e$ is either identical to its gap or to the closure of the complement of its gap. Together with \Cref{lem:strong-c-mon-x-mon} this implies the following central property of strongly c-monotone drawings.

\begin{corollary}\label{cor:strong-c-mon-2-cases}
Let $\mathcal{D}$ be a strongly c-monotone drawing. Then either all gap edges are contained in their respective gaps or $\mathcal{D}$ is strongly isomorphic to an $x$-monotone drawing.
\end{corollary}

\begin{proof}
Assume there exists a gap edge $e$ that is not contained in its gap. Then the wedge $\Lambda_{e}$ of $e$ is the closure of the complement of its gap and hence contains all vertices of $\mathcal{D}$. Consequently, by \Cref{lem:strong-c-mon-x-mon}, $\mathcal{D}$ is strongly isomorphic to an $x$-monotone drawing.
\end{proof}

\Cref{cor:strong-c-mon-2-cases} implies that for every strongly c-monotone drawing that is not strongly isomorphic to an $x$-monotone drawing, its gap edges form a crossing-free Hamiltonian cycle.

\begin{theorem}\label{thm:all-hp-strong-c-mon}
Let $\mathcal{D}$ be a strongly c-monotone drawing of $K_{n}$ and let $v_{a}$ and $v_{b}$ be two vertices in~$\mathcal{D}$. Then $\mathcal{D}$ contains a crossing-free Hamiltonian path with end-vertices $v_{a}$ and $v_{b}$.
\end{theorem}

\begin{proof}
If $\mathcal{D}$ is strongly isomorphic to an $x$-monotone drawing, then the statement follows from \Cref{prop:all-hp-x-mon}. So we can assume, by \Cref{cor:strong-c-mon-2-cases}, that all gap edges are contained in their respective gaps.

We distinguish two cases. If $v_{a}$ and $v_{b}$ are neighbors in the circular order around $O$, then the Hamiltonian path between them that consists only of gap edges is crossing-free. Otherwise, let $v_{a}'$ and $v_{b}'$ be the vertices in clockwise circular order around $O$ directly before $v_{a}$ and~$v_{b}$, respectively; see \Cref{fig:c-mon-all-pairs_a} for an illustration. Let further, without loss of generality, the wedge $\Lambda_{e}$ of the edge $e = \{ v_{a}', v_{b}' \}$ contain $v_{a}$. Then, by \Cref{lem:strong-c-mon-x-mon}, the sub-drawing $\mathcal{D}_{e}$ induced by the vertices in $\Lambda_{e}$ is strongly isomorphic to an $x$-monotone drawing. Hence, by \Cref{prop:all-hp-x-mon}, there exists a crossing-free Hamiltonian path in $\mathcal{D}_{e}$ with end-vertices $v_{a}$ and $v_{a}'$. Finally, extending that path from $v_{a}'$ to $v_{b}$ by gap edges in counter-clockwise order around $O$ creates a Hamiltonian path with end-vertices $v_{a}$ and $v_{b}$ in~$\mathcal{D}$ which, by \Cref{lem:strong-c-mon-x-mon} and the above assumption, is crossing-free.
\end{proof}

\begin{figure}[h]
\centering
\subcaptionbox{\centering\label{fig:c-mon-all-pairs_a}}[.49\textwidth]{\includegraphics[page=3]{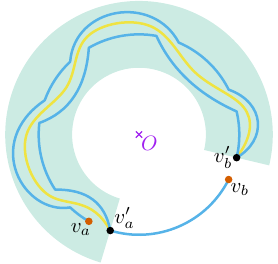}}
\subcaptionbox{\centering\label{fig:twisted-all-pairs}}[.49\textwidth]{\includegraphics[page=4]{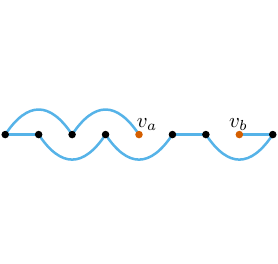}}
\caption{\textbf{(a)}~Construction of a crossing-free Hamiltonian path between two given vertices in a strongly c-monotone drawing of $K_{n}$. \textbf{(b)}~A crossing-free Hamiltonian path between two given vertices of a twisted drawing. For better readability, the edges are not drawn in the usual way but instead indicated figuratively.}
\label{fig:c-mon-all-pairs}
\end{figure}

By \Cref{cor:strong-c-mon-conjectures}, this also proves \Cref{conj:main} for strongly c-monotone drawings.

\begin{corollary}\label{cor:cfhc-strong-c-mon}
Let $\mathcal{D}$ be a strongly c-monotone drawing of $K_{n}$ with $n \geq 3$. Then $\mathcal{D}$ contains a crossing-free Hamiltonian cycle.
\end{corollary}

Let us mention at this point that \Cref{conj:main,conj:stronger} also hold for the twisted drawing~$\mathcal{T}_{n}$. For constructing a crossing-free Hamiltonian path or cycle, we just have to make sure not to use any pair of nested edges. \Cref{fig:twisted-all-pairs} indicates a possible crossing-free Hamiltonian path between two given vertices in a twisted drawing, using only edges between vertices that are at most at distance two from each other in the defining vertex order.

In \Cref{fig:c-mon-example} we give an example of a strongly c-monotone drawing that is not weakly isomorphic to any $x$-monotone or cylindrical drawing, by \Cref{obs:x-bound-uncrossed,lem:cylindrical-rim}, because it does not have any completely uncrossed edges. It is also not g-convex since it contains triangles without a convex side. In particular, this means that \Cref{cor:cfhc-strong-c-mon} verifies \Cref{conj:main} for more simple drawings than the ones for which it was known before.

\begin{figure}[t]
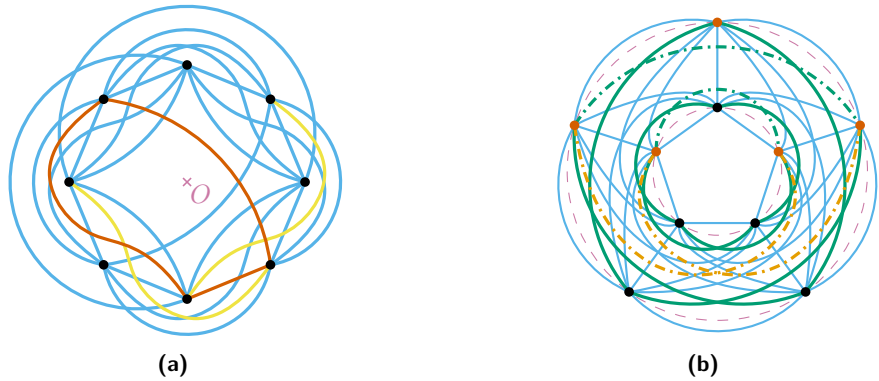

\centering
\subcaptionbox{\centering\label{fig:c-mon-example}}[.49\textwidth]{\includegraphics[page=4]{Figures/special_c_mon.pdf}}
\subcaptionbox{\centering\label{fig:cylindrical-example}}[.49\textwidth]{\includegraphics[page=5]{Figures/special_c_mon.pdf}}
\caption{\textbf{(a)}~A strongly c-monotone drawing of $K_{8}$ that is neither $x$-monotone nor cylindrical nor g-convex. The darkorange triangle has no convex side, which is witnessed by the two yellow edges. \textbf{(b)}~A cylindrical drawing $\mathcal{D}_{10}$ of $K_{10}$ that is neither strongly c-monotone nor g-convex.}
\label{fig:c-mon-cylindrical-examples}
\vspace{-0.5\baselineskip}
\end{figure}

Furthermore, \Cref{fig:cylindrical-example} shows an example of a cylindrical drawing $\mathcal{D}_{10}$ that is neither \mbox{g-}convex nor weakly isomorphic to any strongly c-monotone drawing. For the not g-convex part, for example, the sub-drawing induced by the five darkorange vertices is weakly isomorphic to the non-straight-line drawing of $K_{5}$ in \Cref{fig:T3}. To show that $\mathcal{D}_{10}$ is not weakly isomorphic to any strongly c-monotone drawing, we need more involved arguments. We give the basic ideas in the following.

First, imagine a c-monotone drawing being mapped onto the sphere such that the origin $O$ and the point at infinity, which we denote by $P_{\infty}$, are opposite poles. Then $O$ and $P_{\infty}$ behave exactly the same with regard to the definition of c-monotonicity and, thus, can be interchanged. Next, we consider sub-structures of simple drawings that restrict where $O$ and $P_{\infty}$ must be placed such that a drawing $\mathcal{D}$ can be c-monotone. The first structure, which we call a \emph{cloud}\footnote{Arroyo, Bensmail, and Richter~\cite{abr-2021-edgap} used similar structures to characterize pseudolinear drawings.}, is a simple closed curve $\mathcal{B}$ consisting of parts of the relative interior of edges in $\mathcal{D}$ such that the end-vertices of all involved edges lie on the same side of $\mathcal{B}$, which we call the \emph{inside} of the cloud, and no end-vertex lies on $\mathcal{B}$. Thereby the sides of $\mathcal{B}$ are defined analogously to the sides of a triangle. \Cref{fig:cloud} shows a cloud that also appears in~$\mathcal{D}_{10}$. The second structure, which we call a \emph{fish}, is defined the same way as a cloud but with one vertex of $\mathcal{D}$ lying on the simple closed curve $\mathcal{B}$. The edges forming the cloud in \Cref{fig:cloud} in fact also form five fishes; one of them is shown in \Cref{fig:fish}. The intersection of the insides of these five fishes, which we call the \emph{center cell} of the cloud, is shaded seagreen in \Cref{fig:cloud}. For a fish or cloud $\mathcal{B}$ in a c-monotone drawing, either $O$ or $P_{\infty}$ must lie in its inside because $\mathcal{B}$ contains at most one vertex, where it could change its circular direction around $O$ without forcing an edge to violate c-monotonicity. If both $O$ and $P_{\infty}$ would lie in the outside, then $\mathcal{B}$ would have to change its direction at least twice. Now, observe that $\mathcal{D}_{10}$ actually contains two such clouds, which are formed by the seagreen edges in \Cref{fig:cylindrical-example} and have disjoint insides. Hence, $O$ must lie in the center cell of one of the two clouds and $P_{\infty}$ must lie in the other. Finally, we consider \emph{pretzels}, which are cycles of length $4$ that encircle some area in the plane twice; shaded orange in \Cref{fig:pretzel}. If, in a c-monotone drawing $\mathcal{D}$, $O$ lies in that area, then some pair of edges of the pretzel will cover the plane and consequently $\mathcal{D}$ is not strongly c-monotone. By the above arguments and a result from Gioan~\cite[Theorem 3.10]{g-2022-cgdtm} stating that weakly isomorphic drawings of $K_{n}$ can be transformed into each other by triangle flips, this is the case for every c-monotone drawing that is weakly isomorphic to $\mathcal{D}_{10}$; the pretzel in \Cref{fig:cylindrical-example} is drawn dash dotted.

\begin{figure}[h]
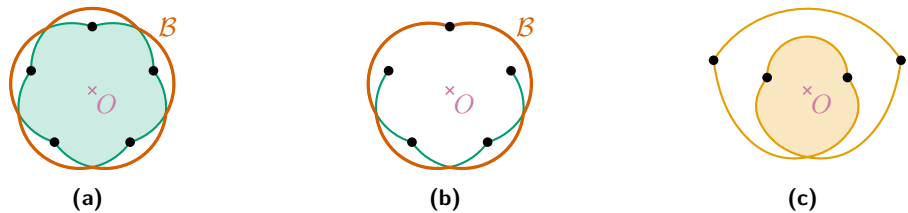

\centering
\subcaptionbox{\centering\label{fig:cloud}}[.328\textwidth]{\includegraphics[page=6]{Figures/special_c_mon.pdf}}
\subcaptionbox{\centering\label{fig:fish}}[.328\textwidth]{\includegraphics[page=7]{Figures/special_c_mon.pdf}}
\subcaptionbox{\centering\label{fig:pretzel}}[.328\textwidth]{\includegraphics[page=8]{Figures/special_c_mon.pdf}}
\caption{The three structures we use to show that $\mathcal{D}_{10}$ from \Cref{fig:cylindrical-example} is not weakly isomorphic to any strongly c-monotone drawing: \textbf{(a)}~A cloud, \textbf{(b)}~a fish, and \textbf{(c)}~a pretzel.}
\label{fig:cloud-fish-pretzel}
\vspace{-\baselineskip}
\end{figure}

Since we now know that not all cylindrical drawings are weakly isomorphic to a strongly c-monotone drawing, we conclude this section by verifying \Cref{conj:stronger} for cylindrical drawings. In our proof, we use a similar idea as in a \enquote{classroom proof} of \Cref{conj:main} for \emph{strongly} cylindrical drawings.

\begin{theorem}\label{thm:all-hp-cylindrical}
Let $\mathcal{D}$ be a cylindrical drawing of $K_{n}$ and let $v_{a}$ and $v_{b}$ be two vertices in~$\mathcal{D}$. Then $\mathcal{D}$ contains a crossing-free Hamiltonian path with end-vertices $v_{a}$ and $v_{b}$.\vspace{-1pt}
\end{theorem}

\begin{proof}
By \Cref{lem:cylindrical-rim}, all but at most one rim edges per circle are completely uncrossed. Moreover, if a rim edge $f$ has crossings, then $f$ lies in the lateral face and guards all vertices on its circle. We now distinguish two situations.

Let first $v_{a}$ and $v_{b}$ lie on different circles of $\mathcal{D}$. See \Cref{fig:cylindrical-all-pairs_a} for an illustration of the following. We construct a crossing-free path $\mathcal{P}_{1}$ by starting at $v_{a}$ and visiting all vertices on the same circle in clockwise order, if possible. If no rim edge is crossed then going along them yields the desired result. If at some point we reach a rim edge $f_{1}$ with crossings, then, instead of~$f_{1}$, we take the edge~$e_{1}$ to the first vertex before $v_{a}$ and continue in counter-clockwise order for the rest of the circle. In the same manner, we construct a crossing-free path $\mathcal{P}_{2}$ starting at $v_{b}$ and visiting all vertices on the second circle, potentially using a circle edge $e_{2}$ instead of a rim edge~$f_{2}$ with crossings. Finally, we connect the end-vertices of $\mathcal{P}_{1}$ and $\mathcal{P}_{2}$ that are different from $v_a$ and $v_b$ (unless the respective path has only one vertex) by a lateral edge~$e$. This produces a Hamiltonian path $\mathcal{P}$ in $\mathcal{D}$ with end-vertices $v_{a}$ and $v_{b}$. Further, $\mathcal{P}$ is crossing-free because $e$, $e_{1}$, and $e_{2}$ are the only edges in $\mathcal{P}$ that could have crossings. However, if $e_{1}$ and/or $e_{2}$ are present, then $f_{1}$ and/or $f_{2}$ partition the lateral face into up to three areas such that $e$ is contained in one area $A$ of them, and $e_{1}$ and $e_{2}$ are contained in two different connected components of the complement of $A$.

\begin{figure}[h]
\centering\vspace{-1pt}
\subcaptionbox{\centering\label{fig:cylindrical-all-pairs_a}}[.49\textwidth]{\includegraphics[page=5]{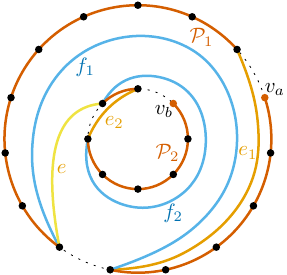}}
\subcaptionbox{\centering\label{fig:cylindrical-all-pairs_b}}[.49\textwidth]{\includegraphics[page=6]{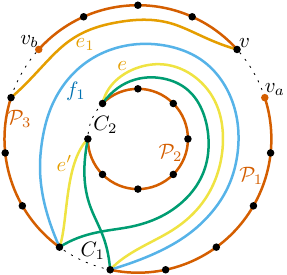}}\vspace{-3pt}
\caption{Construction of a crossing-free Hamiltonian path between two given vertices $v_{a}$ and $v_{b}$ in a cylindrical drawing when $v_{a}$ and $v_{b}$ lie on \textbf{(a)}~different circles or \textbf{(b)}~the same circle. In both situations the case with a crossed rim edge is shown.}
\label{fig:cylindrical-all-pairs}
\vspace{-2pt}
\end{figure}

Let next $v_{a}$ and $v_{b}$ lie on the same circle $C_{1}$ and assume that there is at least one vertex on the other circle $C_{2}$. See \Cref{fig:cylindrical-all-pairs_b} for an illustration. By \Cref{lem:cylindrical-rim}, there exists a path $\mathcal{P}_{2}$ consisting of completely uncrossed rim edges that visits all vertices of $C_{2}$. For connecting the remaining vertices, assume first that all rim edges on $C_{1}$ are completely uncrossed. Let then $\mathcal{P}_{1}$ be the path of completely uncrossed rim edges starting at $v_{a}$ and visiting all vertices in clockwise order on $C_{1}$ until one vertex before $v_{b}$. Accordingly, let $\mathcal{P}_{3}$ be the path of completely uncrossed rim edges starting at $v_{b}$ and visiting all vertices in clockwise order on $C_{1}$ until one vertex before $v_{a}$. By this, $\mathcal{P}_{1}$ and $\mathcal{P}_{3}$ cover all vertices of~$C_{1}$. In the remaining case, when there exists a unique rim edge $f_{1}$ on $C_{1}$ with crossings, we assume without loss of generality that $f_{1}$ lies between $v_{a}$ and $v_{b}$ in clockwise direction along~$C_{1}$. Let then $\mathcal{P}_{1}$ be the completely uncrossed path starting at $v_{a}$ and visiting all vertices in clockwise order on $C_{1}$ until the first end-vertex of~$f_{1}$. For the path $\mathcal{P}_{3}$, we start at $v_{b}$ and visit all vertices in clockwise order on $C_{1}$ via completely uncrossed rim edges until the last vertex $v$ before $v_{a}$. If~$v_{b}$ is the second end-vertex of $f_{1}$, then $\mathcal{P}_{1}$ and $\mathcal{P}_{3}$ already cover all vertices of $C_{1}$. Otherwise, we extend $\mathcal{P}_{3}$ by the edge $e_{1}$ from $v$ to the last vertex before $v_{b}$ in clockwise order on $C_{1}$ and continue from there in counter-clockwise order along $C_{1}$, again via completely uncrossed rim edges, until we reach the second end-vertex of $f_{1}$. We then connect the three paths $\mathcal{P}_{1}$, $\mathcal{P}_{2}$, and $\mathcal{P}_{3}$ with two lateral edges $e$ and $e'$ to a Hamiltonian path $\mathcal{P}$ with end-vertices $v_{a}$ and~$v_{b}$. In particular, there are two choices on how to connect the end-vertices of $\mathcal{P}_{1}$ and $\mathcal{P}_{3}$ that are different from $v_{a}$ and $v_{b}$ (unless the respective path has only one vertex) with the end-vertices of $\mathcal{P}_{2}$ (unless $\mathcal{P}_{2}$ has only one vertex, but then the unique choice is crossing-free). At least one of those choices is a non-crossing edge pair $e$ and~$e'$ because there can be at most one crossing induced by any $4$-tuple of vertices. Consequently, the only potential crossings in $\mathcal{P}$ are between the connection edges $e$ and $e'$, and the non-rim edge $e_{1}$, if it exists. However, since in that case $f_{1}$ splits the lateral face into two areas such that $e_{1}$ cannot enter the area containing $e$ and $e'$, $\mathcal{P}$ is again crossing-free.

Finally, if all vertices lie on one circle, then, by \Cref{obs:cylindrical-2-page}, $\mathcal{D}$ is strongly isomorphic to a $2$-page-book drawing. Since every $2$-page-book drawing is weakly isomorphic to an x-monotone drawing, the statement follows by \Cref{prop:all-hp-x-mon}. This completes the proof.
\end{proof}

\vspace{-1pt}
Again, by \Cref{cor:cylindrical-conjectures}, this implies that also \Cref{conj:main} holds for cylindrical drawings.

\begin{corollary}\label{cor:cfhc-cylindrical}
Let $\mathcal{D}$ be a cylindrical drawing of $K_{n}$ with $n \geq 3$. Then $\mathcal{D}$ contains a crossing-free Hamiltonian cycle.
\end{corollary}

\section{Relations between classes of simple drawings}\label{sec:relations}

\vspace{-1pt}
This section is devoted to analyzing inclusion relations between different classes of simple drawings, in particular, those that are marked darkorange in \Cref{fig:the-class-order}.
\vspace{-2pt}

\subsection{\texorpdfstring{$X$}{X}-bounded drawings}\label{sec:x-bound-relate}
\vspace{-1pt}

We start by showing that every $x$-bounded drawing $\mathcal{D}$ of $K_{n}$ is weakly isomorphic to an \mbox{$x$-}mono\-tone drawing. This also follows from a result by Balko, Fulek, and Kyn{\v{c}}l~\cite[Lemma~4.8]{bfk-2015-cnccmd}. However, since their proof relies on many other results in that paper, which makes the ideas behind it hard to grasp, we present a self-contained proof here. The main ingredient for our proof is given by the following lemma.

\begin{lemma}\label{lem:x-bound-main}
Let $e = \{ v_{a}, v_{b} \}$ be an edge and $v$ a vertex with $x$-coordinate between $v_{a}$ and $v_{b}$ in an $x$-bounded drawing of $K_{n}$. Then $e$ crosses the vertical line through $v$ either above $v$ or below $v$, at least once, but never on both sides.
\end{lemma}

\begin{proof}
Recall from \Cref{sec:all-pair-proofs} that $e$ splits the vertical strip between $v_{a}$ and $v_{b}$ into two sides, one above $e$ and one below $e$. We assume, without loss of generality, that $v$ lies below $e$ and that $v_{a}$ is left of $v_{b}$; see \Cref{fig:special-x-bound_a} for an illustration. Further, we consider the edges $f_{1} = \{ v_{a}, v \}$ and $f_{2} = \{ v, v_{b} \}$. By \Cref{obs:x-bound-edge-above-below}, both $f_1$ and $f_2$ lie completely below $e$ and, by $x$-boundedness, $f_{1}$ lies left of the vertical line through $v$ and $f_{2}$ lies to its right. Hence, the union of $f_{1}$ and $f_{2}$ splits the vertical strip between $v_{a}$ and $v_{b}$ into two parts, an upper and a lower part, such that $e$ is contained in the upper part and the vertical ray that starts in $v$ and goes downwards is contained in the lower part. Consequently, $e$ only crosses the vertical line through $v$ above $v$, which it must cross at least once to connect $v_{a}$ and $v_{b}$.
\end{proof}

\begin{figure}[h]
\centering
\subcaptionbox{\centering\label{fig:special-x-bound_a}}[.49\textwidth]{\includegraphics[page=1]{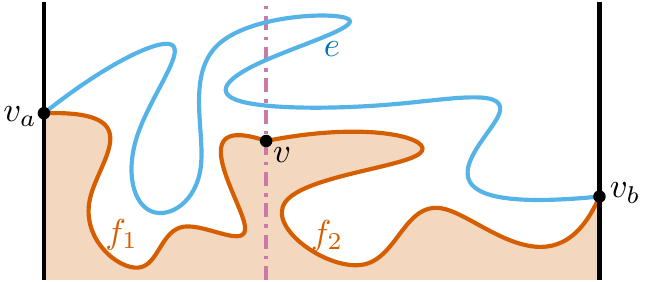}}
\subcaptionbox{\centering\label{fig:special-x-bound_b}}[.49\textwidth]{\includegraphics[page=2]{Figures/special_x_bound.pdf}}
\caption{\textbf{(a)}~The edge $e$ can only cross the vertical violet line through $v$ on one side, above or below, because the other side lies in the orange area bounded by $f_{1}$ and $f_{2}$. \textbf{(b)}~The edge $e$ crosses the vertical line through $v$ on both sides; the dashed darkorange edge $f_{2} = \{ v, v_{b} \}$ cannot be inserted anymore within the vertical strip between $v$ and $v_{b}$ without crossing its incident edge~$e$.}
\label{fig:special-x-bound}
\end{figure}

It is crucial for the proof that both edges $f_{1}$ and $f_{2}$ exist in the drawing. In the other direction, if the edge $e$ crosses the vertical line through $v$ on both sides, then at least one of those two edges cannot be added in an $x$-bounded way anymore; see \Cref{fig:special-x-bound_b}. Hence, \Cref{lem:x-bound-main} and subsequent results only hold for $x$-bounded drawings of complete graphs.

From here on the arguments are mostly of technical nature. We start by introducing, for each vertex $v$ of an $x$-bounded drawing $\mathcal{D}$, a partial order $<_{v}$ on the set of edges in $\mathcal{D}$, defined by the following four conditions for $e <_{v} f$:

\begin{itemize}
\item $e$ and $f$ are incident to $v$, and $e$ leaves $v$ below $f$ on the same side (left or right),
\item $e$ crosses the vertical line through $v$ below $v$ and $f$ is incident to $v$,
\item $f$ crosses the vertical line through $v$ above $v$ and $e$ is incident to $v$, or
\item the vertical line through $v$ is crossed below $v$ by $e$ and above $v$ by $f$.
\end{itemize}

For a fixed vertex $v$, any of these four conditions potentially induces an order between two edges. However, by \Cref{lem:x-bound-main}, for each pair of edges at most one of the four conditions holds and the forth condition is well-defined. Note that there is no relation between two edges $e$ and $f$ if either they both cross the vertical line through $v$ on the same side of $v$ or one lies completely to the left of $v$ and the other lies completely to the right of $v$. Further, it can be verified that antisymmetry and transitivity are fulfilled. Since no edge gets related to itself, the conditions in fact induce a partial order. In \Cref{fig:special-x-bound_2} we show two examples of such a partial order of edges at a vertex.

\begin{figure}[h]
\centering
\subcaptionbox{\centering\label{fig:special-x-bound_2a}}[.40\textwidth]{\includegraphics[page=3]{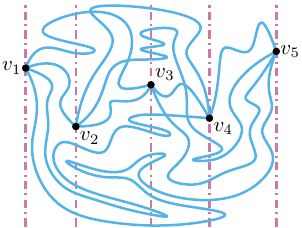}}
\subcaptionbox{\centering\label{fig:special-x-bound_2b}}[.29\textwidth]{\includegraphics[page=4]{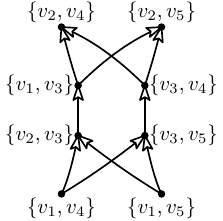}}
\subcaptionbox{\centering\label{fig:special-x-bound_2c}}[.29\textwidth]{\includegraphics[page=5]{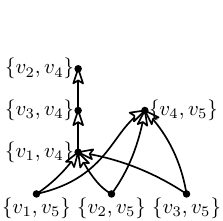}}
\caption{\textbf{(a)}~A realization of the twisted drawing $\mathcal{T}_{5}$ as a quite wiggly $x$-bounded drawing; the violet lines mark the bounds for the edges. \textbf{(b)}~A Hasse diagram for the partial order~$<_{v_{3}}$ of the drawing in (a). \textbf{(c)}~A Hasse diagram for the partial order~$<_{v_{4}}$ of the drawing in (a).}
\label{fig:special-x-bound_2}
\end{figure}

We next discuss how to determine all crossings of an $x$-bounded drawing of $K_{n}$ by those partial orders. Recall that we also have a total order $<$ on the vertices from left to right. In particular, for $v < w$, the partial orders $<_{v}$ and $<_{w}$ determine, for two comparable edges, in which order from bottom to top they enter and leave, respectively, the vertical strip between $v$ and $w$. See \Cref{fig:special-x-bound_3a} for an example illustration of the following observation.

\begin{observation}\label{obs:x-bound-crossings-1}
Let $\mathcal{D}$ be an $x$-bounded drawing of $K_{n}$. If, for two edges $e$ and $f$ and vertices $v < w$ in $\mathcal{D}$, the inequalities $e <_{v} f$ and $e >_{w} f$ hold, then $e$ and $f$ have a crossing in the vertical strip between $v$ and $w$.
\end{observation}

Especially, if two edges are nested or linked, then they are always comparable at exactly two of their four end-vertices. Hence, using \Cref{obs:x-bound-crossings-1} and extending \Cref{obs:x-bound-edge-above-below}, we can classify in which pattern two edges $e$ and $f$ have to pass below or above each others end-vertices to form a crossing, depending on whether $e$ and $f$ are nested or linked; if they are separated, they cannot cross anyway. See \Cref{fig:special-x-bound_3b} for an illustration of some cases from the following observation.

\begin{observation}\label{obs:x-bound-crossings-2}
Let $\mathcal{D}$ be an $x$-bounded drawing of $K_{n}$ with vertices $v_{1}, \ldots, v_{n}$ from left to right. Let $e = \{ v_{a}, v_{b} \}$ and $f = \{ v_{c}, v_{d} \}$ be two edges in~$\mathcal{D}$ with $a \leq c$ and, by convention, $a < b$ and $c < d$. Then $e$ and $f$ cross if and only if one of the following two conditions holds:
\begin{itemize}
\item $e$ and $f$ are nested, and ($e <_{v_{c}} f$ and $e >_{v_{d}} f$) or ($e >_{v_{c}} f$ and $e <_{v_{d}} f$); or
\item $e$ and $f$ are linked, and ($e <_{v_{c}} f$ and $e >_{v_{b}} f$) or ($e >_{v_{c}} f$ and $e <_{v_{b}} f$).
\end{itemize}
\end{observation}

\begin{figure}[h]
\centering
\subcaptionbox{\centering\label{fig:special-x-bound_3a}}[.32\textwidth]{\includegraphics[page=6]{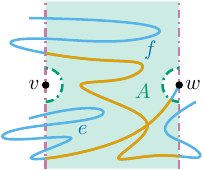}}
\subcaptionbox{\centering\label{fig:special-x-bound_3b}}[.66\textwidth]{\includegraphics[page=7]{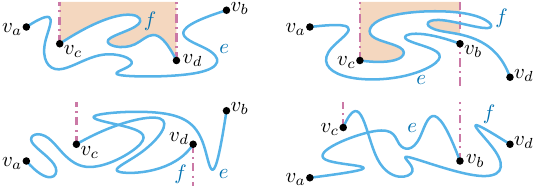}}
\caption{\textbf{(a)}~Illustration of \Cref{obs:x-bound-crossings-1}: The orange parts of $e$ and $f$ (between the last entry point from the left and the first exit point after that to the right) must cross within the seagreen area $A$ (the vertical strip between $v$ and $w$ excluding a small $\varepsilon$-ball around $v$ and $w$, each). \textbf{(b)}~Illustrations of \Cref{obs:x-bound-crossings-2}: Two nested cases on the left and two linked cases on the right. In the bottom cases the edges cross by \Cref{obs:x-bound-crossings-1}. In the top cases the edges cannot cross because $e$ cannot enter the darkorange areas above $f$ without crossing $f$ at least twice.}
\label{fig:special-x-bound_3}
\end{figure}

We stated the above observation for $a \leq c$ to also cover the case of incident edges, which cannot cross anyway. With this, we are ready to prove that every $x$-bounded drawing of $K_{n}$ is weakly isomorphic to an $x$-monotone drawing. This is similar in spirit to a result shown by Fulek, Pelsmajer, Schaefer, and {\v{S}}tefan{\-}kovi{\v{c}}~\cite[Lemma 2.6]{fpss-2013-htmdlp}: Every $x$-bounded drawing can be made $x$-monotone without changing the parity of crossings between any pair of edges or changing the rotation around any vertex. While their result also holds for drawings of non-complete graphs, however, in their setting the drawings need not be simple. In particular, even if the initial $x$-bounded drawing is simple, the resulting $x$-monotone drawing after applying their transformation might not be simple anymore.

\begin{theorem}\label{thm:x-bound-is-x-mon}
Let $\mathcal{D}$ be an $x$-bounded drawing of $K_{n}$ with vertices $v_{1}, \ldots, v_{n}$ in that order from left to right. Then there exists an $x$-monotone drawing $\mathcal{D}'$ that is weakly isomorphic to $\mathcal{D}$ and has the same vertex order.
\end{theorem}

\begin{proof}
We will construct an $x$-monotone drawing $\mathcal{D}'$ that has the exact same set of crossing edge pairs as $\mathcal{D}$ and the same vertex order $v_{1}, \ldots, v_{n}$ from left to right; see \Cref{fig:x-bound-to-x-mon} for an example illustration of the following steps.

We place the vertices $v_{1}, \ldots, v_{n}$ in that order on a horizontal line. For the edges, we consider the vertical strips between the vertices from left to right and, in each strip, add all edges simultaneously as follows. We define orders on where the edges enter the strip from the left as well as where they leave the strip to the right, from bottom to top in both cases. Subsequently, we connect all the respective entry and exit points by line segments.

In detail, we inductively consider the vertical strip between and including vertices $v_{i}$ and~$v_{i+1}$, for $i=1,\ldots, n-1$. We can assume that all entry points from the left of edges crossing the vertical line through $v_{i}$ below or above $v_{i}$ are given by the exit points to the right in the vertical strip between vertices $v_{i-1}$ and $v_{i}$. These entry points induce a linear order of those edges from bottom to top, which we denote by~$<_{i}^{\leftarrow}$; the order on the left boundary of the strip between $v_{i}$ and $v_{i+1}$. For the first strip between vertices $v_{1}$ and $v_{2}$, there are no such edges entering from the left, that is, the order is empty at this point. In any case, it remains to add the edges incident to $v_{i}$ to this order.

For that, we insert all the edges incident to $v_{i}$ that in $\mathcal{D}$ leave $v_{i}$ to the right into~$<_{i}^{\leftarrow}$ between the edges crossing the vertical line through $v_{i}$ below $v_{i}$ and those crossing the line above $v_{i}$. Moreover, we add the edges ordered according to the order $<_{v_{i}}$ in $\mathcal{D}$, which is the same as the order given by the counter-clockwise rotation of these edges at $v_i$. Hence the resulting order $<_{i}^{\leftarrow}$ agrees with the partial order $<_{v_{i}}$ in $\mathcal{D}$.

Next we create an order $<_{i}^{\rightarrow}$; the order on the right boundary of the strip between $v_{i}$ and~$v_{i+1}$. To this end, we split the edges into three groups, namely, the ones leaving the strip below $v_{i+1}$, the ones ending in $v_{i+1}$, and the ones leaving the strip above~$v_{i+1}$. By \Cref{lem:x-bound-main}, these three groups are well-defined. Also note that the edges in each of those three groups are in general not consecutive in $<_{i}^{\leftarrow}$. To obtain $<_{i}^{\rightarrow}$, we start with all edges leaving the strip below $v_{i+1}$, continue with all edges ending in $v_{i+1}$, and finish with all edges leaving the strip above $v_{i+1}$, in each of the three groups keeping the relative order between the edges as given by $<_{i}^{\leftarrow}$. With that, for two edges in different of those three groups, $<_{i}^{\rightarrow}$~agrees with the partial order $<_{v_{i+1}}$ in~$\mathcal{D}$.

\begin{figure}[h]
\centering
\subcaptionbox{\centering\label{fig:x-bound-to-x-mon-steps}}[.67\textwidth]{\includegraphics[page=8]{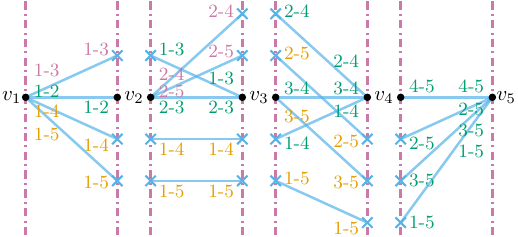}}
\subcaptionbox{\centering\label{fig:x-bound-to-x-mon-final}}[.31\textwidth]{\includegraphics[page=9]{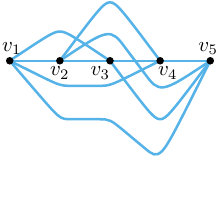}}
\caption{\textbf{(a)}~The four steps of redrawing the $x$-bounded drawing from \Cref{fig:special-x-bound_2a} strip by strip into a weakly isomorphic $x$-monotone drawing. Note that each of the vertices $v_2$ to $v_4$ appears in two steps. The placement of the pairs \enquote{a-b}, for edges $\{ v_{a}, v_{b} \}$, from bottom to top corresponds to the orders~$<_{i}^{\leftarrow}$ and $<_{i}^{\rightarrow}$ on the left and right boundary of each strip, respectively. The colors of the pairs indicate the three groups, into which the edges are placed in $<_{i}^{\rightarrow}$: Orange for edges leaving the strip below $v_{i+1}$, seagreen for edges ending in $v_{i+1}$, and violet for edges leaving the strip above~$v_{i+1}$. \textbf{(b)}~The final result, with the edges smoothed at the transition points between the strips.}
\label{fig:x-bound-to-x-mon}
\end{figure}

Finally, we mark the exit points of the edges on the vertical line through $v_{i+1}$, in the order given by $<_{i}^{\rightarrow}$ from bottom to top, such that every edge passing below or above $v_{i+1}$ gets its individual exit point, while all edges ending in $v_{i+1}$ share the position of $v_{i+1}$ as their exit point. For each edge, we connect the corresponding entry point on the left boundary of the strip with the exit point on the right boundary of the strip by a line segment. Then two of those line segments, belonging to edges $e$ and~$f$, respectively, cross in $\mathcal{D}'$ if and only if the order of $e$ and $f$ changes between $<_{i}^{\leftarrow}$ and $<_{i}^{\rightarrow}$. Without loss of generality, let $e <_{i}^{\leftarrow} f$~and~$f <_{i}^{\rightarrow} e$. We next show, by applying \Cref{obs:x-bound-crossings-1}, that this crossing between the line segments belonging to $e$ and~$f$ in $\mathcal{D}'$ uniquely corresponds to a crossing of $e$~and~$f$~in~$\mathcal{D}$.

Since, within each of the three groups on the right boundary of the strip, we keep the relative order from the left boundary of the strip, $e$ and $f$ must have been placed into different groups on the right. Therefore, $f <_{i}^{\rightarrow} e$ implies that also $f <_{v_{i+1}} e$ holds.

Regarding the partial order $<_{v_{i}}$ the edges $e$ and $f$ might be incomparable. However, there must be a vertex $w \leq v_{i}$ such that $e$ and $f$ are comparable with respect to $<_{w}$; this is at least the case for the start-vertex of either $e$ or~$f$. Let $k \leq i$ be maximal such that $e$ and $f$ are comparable with respect to $<_{v_{k}}$. Hence, for all $k < j \leq i$, the edges $e$ and $f$ both cross the vertical line through $v_{j}$ on the same side of $v_{j}$. Since $<_{j}^{\leftarrow}$ agrees with~$<_{j-1}^{\rightarrow}$ for edges crossing above or below $v_{j}$, and $<_{j-1}^{\rightarrow}$ agrees with $<_{j-1}^{\leftarrow}$ for edges placed in the same group on the right boundary of the strip, $e <_{i}^{\leftarrow} f$ implies that $e <_{k}^{\leftarrow} f$ holds. Consequently, since $e$ and $f$ are comparable with respect to $<_{v_{k}}$, which agrees with $<_{k}^{\leftarrow}$, we get $e <_{v_{k}} f$. Hence, by \Cref{obs:x-bound-crossings-1}, $e$ and $f$ cross in the vertical strip between vertices $v_{k}$ and $v_{i+1}$ in~$\mathcal{D}$.

Moreover, $e$ and $f$ are comparable with respect to $<_{v_{i+1}}$. So any potential further crossing between $e$ and $f$ in $\mathcal{D}'$ would correspond to a crossing between $e$ and $f$ in $\mathcal{D}$ that lies to the right of the vertical line through $v_{i+1}$. This cannot exist because $\mathcal{D}$ is simple.

It remains to argue that every crossing in $\mathcal{D}$ also exists in $\mathcal{D}'$. By \Cref{obs:x-bound-crossings-2}, we know that every crossing in $\mathcal{D}$ is in one-to-one correspondence with a change of the order of the involved non-incident edges $e$ and $f$ between two partial orders $<_{v_{i}}$ and $<_{v_{j}}$ with $i<j$, that is, at two of the end-vertices of $e$ and $f$. This change implies that also the orders $<_{i}^{\leftarrow}$ and $<_{j}^{\rightarrow}$ change accordingly, which produces a crossing in the construction of~$\mathcal{D}'$.
\end{proof}

In our construction, edges cross at the latest possible moment, that is, in the rightmost strip of the area given by \Cref{obs:x-bound-crossings-1}. Also, we implicitly use \Cref{lem:x-bound-main} all the time because the orders $<_{v_{i}}$ would not be well defined otherwise. Hence, similar to \Cref{lem:x-bound-main}, \Cref{thm:x-bound-is-x-mon} does not hold for drawings of non-complete graphs. As an example, \Cref{fig:x-bound-non-complete_a} depicts an $x$-bounded drawing $\mathcal{D}_{b}$ for which, as we argue below, no weakly isomorphic $x$-monotone drawing $\mathcal{D}_{m}$ exists. That is, there is no $x$-monotone drawing with the same set of crossing edge pairs as in $\mathcal{D}_{b}$.

Assume for a contradiction that $\mathcal{D}_{m}$ exists. Note that Gioan's theorem~\cite[Theorem 3.10]{g-2022-cgdtm} does not hold for drawings of non-complete graphs. In particular, $\mathcal{D}_{b}$ and $\mathcal{D}_{m}$ might not even have the same rotation system, and structures like clouds or fishes that exist in $\mathcal{D}_{b}$ might not exist in $\mathcal{D}_{m}$ and vice versa. So, as we are not aware of any simpler method to show this, we check all $360$ possible orders, that is, permutations modulo reflection, of the $6$ vertices along the $x$-axis and argue in each case that the order is not possible for $\mathcal{D}_{m}$. In the following we explain the main steps and ideas, which might be of value on their own, but we leave the details to the interested reader.

We start with some purely combinatorial arguments to reduce the number of orders that we have to check in detail. As in \Cref{obs:x-bound-uncrossed}, neither the first nor the last two vertices in the order are allowed to induce an edge with crossings. For example, the order $(v_{1},v_{2},v_{3},v_{5},v_{4},v_{6})$ is not possible because the edge $\{ v_{4}, v_{6} \}$ needs to cross the edge $\{ v_{3}, v_{5} \}$. This kind of argument alone already leaves only $64$ possible orders. More generally, two separated edges can never cross in an $x$-monotone drawing. Hence, for example, the order $(v_{1},v_{2},v_{4},v_{3},v_{5},v_{6})$ is not possible because the edge $\{ v_{2}, v_{4} \}$ needs to cross the edge $\{ v_{3}, v_{5} \}$. This further narrows it down to $40$ potential orders. Furthermore, if an edge $e$ crosses a triangle $\Delta$ an odd number of times, then one of the end-vertices of $e$ must lie inside $\Delta$ and the other one outside. In particular, it is not possible that all three vertices of $\Delta$ lie between the end-vertices of $e$ in the $x$-monotone order because then both end-vertices of $e$ would definitely lie outside of $\Delta$. For example, the triangle $\{ v_{2}, v_{4}, v_{5} \}$ is crossed once by the edge $\{ v_{3}, v_{6} \}$, which forbids the order $(v_{1},v_{3},v_{2},v_{4},v_{5},v_{6})$. This eliminates $19$ more cases.

To rule out the remaining $21$ cases we use \Cref{obs:x-bound-crossings-2}, which establishes a connection between crossings, and edges passing below and above each others end-vertices. This observation also holds for non-complete graphs in the case of $x$-monotone drawings; the graph being complete in the case of $x$-bounded drawings is only needed to have the orders $<_{v}$ well defined. As an example, we consider the order $(v_{1},v_{2},v_{3},v_{4},v_{5},v_{6})$; see \Cref{fig:x-bound-non-complete_b} for an illustration. We first add the edge $e = \{ v_{1}, v_{6} \}$ such that, without loss of generality, $v_{2}$ lies below $e$. Then $v_{3}$ has to lie above $e$ because $\{ v_{2}, v_{3} \}$ crosses~$e$, $v_{4}$ also has to lie above $e$ because $\{ v_{3}, v_{4} \}$ does not cross $e$, and $v_{5}$ has to lie below $e$ again because $\{ v_{4}, v_{5} \}$ crosses~$e$. Further, the edge $\{ v_{2}, v_{6} \}$ passes below $v_{4}$ and does not cross $\{ v_{4}, v_{5} \}$, so it has to pass below $v_{5}$ as well. Finally, the edge $\{ v_{1}, v_{5} \}$ has to pass below $v_{2}$ to cross $\{ v_{2}, v_{6} \}$, but then it cannot cross $\{ v_{2}, v_{3} \}$ anymore; a contradiction. Applying a similar reasoning for the other $20$ cases then finishes the proof that there is no $x$-monotone drawing~$\mathcal{D}_{m}$ being weakly isomorphic to $\mathcal{D}_{b}$.

Let us remark that removing the edges $\{ v_{3}, v_{6} \}$ and $\{ v_{4}, v_{6} \}$ from $\mathcal{D}_{b}$ yields a sub-drawing $\mathcal{D}_{b}'$ for which a weakly isomorphic $x$-monotone drawing $\mathcal{D}_{m}'$ exists. It is reached by changing the position of the edges $\{ v_{2}, v_{4} \}$ and $\{ v_{3}, v_{5} \}$ in the rotation of their end-vertices though; see \Cref{fig:x-bound-non-complete_c}. This also emphasizes that for drawings of non-complete graphs two different rotation systems can produce the same set of crossing edge pairs.

Further, note that in $\mathcal{D}_{b}$ only the three edges $\{ v_{1}, v_{3} \}$, $\{ v_{1}, v_{4} \}$, and $\{ v_{5}, v_{6} \}$ are missing, which cannot be added in an $x$-bounded way anymore once $\{ v_{1}, v_{6} \}$ is fixed; recall \Cref{fig:special-x-bound_b}. All three of them can easily be added to create a (general) simple drawing of $K_{6}$ though.

\begin{figure}[h]
\centering
\subcaptionbox{\centering\label{fig:x-bound-non-complete_a}}[.36\textwidth]{\includegraphics[page=10]{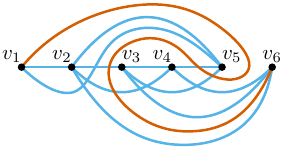}}
\subcaptionbox{\centering\label{fig:x-bound-non-complete_b}}[.31\textwidth]{\includegraphics[page=11]{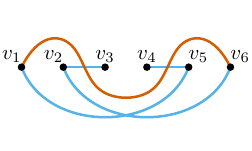}}
\subcaptionbox{\centering\label{fig:x-bound-non-complete_c}}[.31\textwidth]{\includegraphics[page=12]{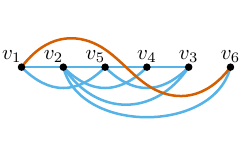}}
\caption{\textbf{(a)} An $x$-bounded drawing $\mathcal{D}_{b}$ of a non-complete graph that is not weakly isomorphic to any $x$-monotone drawing; the darkorange edge is $x$-bounded but not $x$-monotone. \textbf{(b)}~An illustration showing that there cannot be any $x$-monotone drawing $\mathcal{D}_{m}$ on the given vertex order being weakly isomorphic to $\mathcal{D}_{b}$. \textbf{(c)}~An $x$-monotone drawing $\mathcal{D}_{m}'$ that is weakly isomorphic to a sub-drawing of $\mathcal{D}_{b}$ but has a different (partial) rotation system.}
\label{fig:x-bound-non-complete}
\end{figure}

We conclude \Cref{sec:x-bound-relate} with an example of a shellable drawing of $K_{11}$ that is not weakly isomorphic to any $x$-bounded drawing. A similar result was already claimed by Balko, Fulek, and Kyn{\v{c}}l~\cite[Theorem~4.9]{bfk-2015-cnccmd} for a shellable drawing $\mathcal{D}_{9}$ of $K_{9}$; see \Cref{fig:n-shell-is-x-mon-kyncl}. However, what they actually proved is that there exists no $x$-monotone drawing that is weakly isomorphic to $\mathcal{D}_{9}$ \emph{and} has the same unbounded cell as $\mathcal{D}_{9}$. In particular, in \Cref{fig:n-shell-is-x-mon-example} we present an $x$-monotone drawing that is weakly isomorphic to $\mathcal{D}_{9}$; the cell corresponding to the original unbounded cell of $\mathcal{D}_{9}$ is shaded orange. Nevertheless, we can extend their example by two vertices to a simple drawing $\mathcal{D}_{11}$ that actually has the claimed property; see \Cref{fig:n-shell-not-x-mon}. In particular, we duplicate the vertices $1$ and $5$ within the original unbounded cell, similar to the construction in \Cref{fig:add-a-vertex}. Clearly $\mathcal{D}_{11}$ is still shellable. Moreover, if we remove the vertices $1$ and $5$ from~$\mathcal{D}_{11}$, we get a drawing $\mathcal{D}_{9}'$ that is weakly isomorphic to $\mathcal{D}_{9}$. Hence, by the result of Balko, Fulek, and Kyn{\v{c}}l, there does not exist any $x$-monotone drawing that is weakly isomorphic to $\mathcal{D}_{11}$ and has its unbounded cell in the union of the unbounded cells of $\mathcal{D}_{9}$ and~$\mathcal{D}_{9}'$; shaded seagreen in \Cref{fig:n-shell-not-x-mon}. On the other hand, this union is exactly the inside of a fish formed by parts of the edges $\{ 1, x \}$, $\{ x, 5 \}$, and $\{ 1', 5' \}$; dash dotted in \Cref{fig:n-shell-not-x-mon}. Similar to c-monotone drawings, the unbounded face in any $x$-monotone drawing has to lie in the inside of every fish. Putting these two arguments together, we conclude that there is no $x$-monotone or, by \Cref{thm:x-bound-is-x-mon}, $x$-bounded drawing that is weakly isomorphic to~$\mathcal{D}_{11}$.

Two more remarks on the above reasoning: First, in the $x$-monotone drawing given in \Cref{fig:n-shell-is-x-mon-example}, the described vertex duplication would destroy $x$-monotonicity because, for example, we would have to place $1'$ inside the orange cell and route the edge $\{ 1', 5 \}$ around vertex $1$. And second, we can repeat the duplication process on $\mathcal{D}_{11}$ to get arbitrarily large shellable drawings that are not weakly isomorphic to any $x$-bounded drawing.

\begin{observation}\label{obs:n-shell-not-x-mon}
For all $n \geq 11$ there exist shellable drawings of $K_{n}$ that are not weakly isomorphic to any $x$-bounded drawing.
\end{observation}

\begin{figure}[h]
\vspace{-0.5\baselineskip}
\centering
\subcaptionbox{\centering\label{fig:n-shell-is-x-mon-kyncl}}[.328\textwidth]{\includegraphics[page=14]{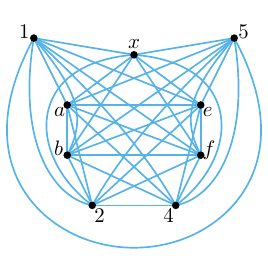}}
\subcaptionbox{\centering\label{fig:n-shell-is-x-mon-example}}[.328\textwidth]{\includegraphics[page=13]{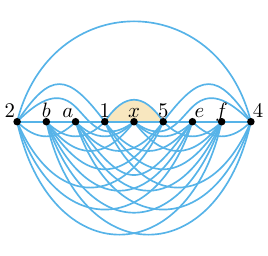}}
\subcaptionbox{\centering\label{fig:n-shell-not-x-mon}}[.328\textwidth]{\includegraphics[page=15]{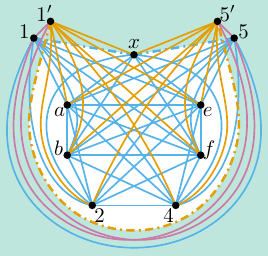}}
\caption{\textbf{(a)}~The original shellable drawing $\mathcal{D}_{9}$ from Balko, Fulek, and Kyn{\v{c}}l~\cite[Figure~21]{bfk-2015-cnccmd} with the vertices labeled as in their paper. \textbf{(b)}~An $x$-monotone drawing that is weakly isomorphic to~$\mathcal{D}_{9}$. \textbf{(c)}~The extended shellable drawing $\mathcal{D}_{11}$, which is not weakly isomorphic to any $x$-bounded drawing anymore. The new edges are orange if they are part of $\mathcal{D}_{9}'$ and purple otherwise.}
\label{fig:n-shell-examples}
\vspace{-0.5\baselineskip}
\end{figure}

\subsection{Cylindrical drawings and c-monotone drawings}\label{sec:cylindrical-relate}

Finally, we turn our attention to the different types of cylindrical drawings. Specifically, we will show that every cylindrical drawing is weakly isomorphic to a c-monotone drawing and that every strongly cylindrical drawing is weakly isomorphic to a strongly c-monotone drawing. Let us emphasize that all results in this section hold for drawings of arbitrary graphs, not just complete graphs. Our first step is to prove that we can redraw a cylindrical drawing $\mathcal{D}$, without changing any crossing properties, such that all edges are c-monotone with respect to the common center~$O$ of the two circles. To this end we define, similar to the winding number of closed curves in complex analysis, the \emph{continuous winding number} $\omega_{e}$ of an edge $e$ as the overall portion of times, as a real number, that $e$ completely travels around $O$ in counter-clockwise direction. To fix the sign of $\omega_{e}$, we orient all lateral edges of $\mathcal{D}$ from the outer to the inner circle and we give all circle edges an arbitrary but fixed orientation. A~negative value of $\omega_{e}$ then means that $e$ travels in clockwise direction around~$O$. To ensure that $\omega_{e}$ is well-defined for all edges $e$, we can assume, without loss of generality, that no inner circle edge actually passes through~$O$. We can give a bound on~$\omega_{e}$, simultaneously for every edge $e$ in $\mathcal{D}$.

\begin{lemma}\label{lem:cylindrical-winding}
Let $\mathcal{D}$ be a cylindrical drawing. Then there exists a cylindrical drawing $\mathcal{D}'$ that is strongly isomorphic to $\mathcal{D}$ and has $\lvert \omega_{e} \rvert < 1$ for every edge $e$ in $\mathcal{D}'$.
\end{lemma}

\begin{proof}
Observe that $\lvert \omega_{e} \rvert < 1$ holds anyway for every circle edge $e$ because otherwise $e$ would have to cross itself. Let further $e_{0}$ and $e_{1}$ be two lateral edges for which $\omega_{e_{0}} = \min_{e \in E} (\omega_{e})$ and $\omega_{e_{1}} = \max_{e \in E} (\omega_{e})$ holds, where $E$ is the set of all lateral edges in $\mathcal{D}$. Observe that $\omega_{e_{1}} - \omega_{e_{0}} < 2$ because otherwise $e_{0}$ and $e_{1}$ would have to cross each other at least twice.

So there exists a homeomorphism of the plane, rotating the outer circle appropriately, that transforms $\mathcal{D}$ into a cylindrical drawing $\mathcal{D}'$ with $\lvert \omega_{e} \rvert < 1$ for every edge $e$ in $\mathcal{D}'$.
\end{proof}

In other words, we can transform every cylindrical drawing $\mathcal{D}$ such that every edge $e$ of $\mathcal{D}$ travels less than one round around $O$, which is a necessary condition for $e$ to be c-monotone with respect to $O$. To further transform $\mathcal{D}$ into a c-monotone drawing, it basically remains to \enquote{stretch} all its edges.

\begin{proposition}\label{prop:cylindrical-is-c-mon}
Let $\mathcal{D}$ be a cylindrical drawing. Then there exists a c-monotone drawing~$\mathcal{D}'$ that is weakly isomorphic to $\mathcal{D}$.
\end{proposition}

\begin{proof}
We argue that all crossings in $\mathcal{D}$ are combinatorially determined by the position of the vertices on the two circles, by the choice of whether an edge lies in the inner, lateral, or outer face, and by the direction an edge travels around $O$.

By \Cref{obs:cylindrical-2-page} the sub-drawings induced by each of the two circles are weakly isomorphic to $2$-page-book drawings. So the crossings between two circle edges on the same circle are determined by \Cref{obs:2-page-crossings}. Further, the crossings between circle edges and lateral edges are determined by \Cref{obs:cylindrical-crossings}. This observation also implies that there are no crossings between two circle edges from different circles. Finally, two lateral edges $e$ and $f$ do \emph{not} cross if and only if $0 \leq \delta + \omega_{f} - \omega_{e} \leq 1$, where $\delta$ is the fraction of the outer circle from the end-vertex of $e$ to the end-vertex of $f$ in counter-clockwise direction.

Since we can also assume, by \Cref{lem:cylindrical-winding}, that $\lvert \omega_{e} \rvert < 1$ holds for every edge $e$ in $\mathcal{D}$, we can easily redraw all edges such that they are c-monotone with respect to $O$ and without changing, for any edge $e$, its $\omega_{e}$ value, the positions of its end-vertices, or the face in which $e$ is drawn. By the above arguments, this results in a c-monotone drawing $\mathcal{D}'$ that is weakly isomorphic to $\mathcal{D}$.
\end{proof}

To show that every strongly cylindrical drawing $\mathcal{D}$ is weakly isomorphic to a strongly c-monotone drawing, we first focus on the lateral edges and assume, by \Cref{prop:cylindrical-is-c-mon}, that $\mathcal{D}$ is c-monotone. Especially, we use the notion of \emph{wedges} and \emph{covering the plane} from c-monotone drawings also for $\mathcal{D}$. Note that two incident lateral edges can never cover the plane because they otherwise would have to cross each other. However, two non-incident lateral edges $e$ and $f$ might cover the plane. In that case the signs of $\omega_{e}$ and $\omega_{f}$ must be the same though because otherwise $e$ and $f$ would cross each other twice. We call such a pair of lateral edges with negative signs a \emph{clockwise double-spiral} and with positive signs a \emph{counter-clockwise double-spiral}; see \Cref{fig:double-spirals_a} for an example.

\begin{lemma}\label{lem:double-spirals}
Let $\mathcal{D}$ be a cylindrical drawing. Then there exists another cylindrical drawing $\mathcal{D}'$ that is weakly isomorphic to $\mathcal{D}$ and contains no double-spirals.
\end{lemma}

\begin{proof}
We will remove one double-spiral at a time by moving around vertices on the inner circle, via a homeomorphism of the plane, without changing their order on the circle. By \Cref{prop:cylindrical-is-c-mon} and its proof, we can assume that the initial drawing $\mathcal{D}$ is c-monotone and we can make each drawing in the process, that results from moving some vertices, again c-monotone without changing the vertex positions.

Let $e = \{ v_{a}, v_{b} \}$ and $f = \{ v_{c}, v_{d} \}$ form a clockwise double-spiral with $v_{a}$ and $v_{c}$ on the outer circle; see \Cref{fig:double-spirals_b} for an illustration. Let further $v_{a}'$ be the neighboring vertex of~$v_{a}$ on the outer circle in counter-clockwise direction. Then we move vertex~$v_{d}$ in counter-clockwise direction on the inner circle out of the wedge $\Lambda_{e}$ and into the wedge $A$ between $v_{a}$ and $v_{a}'$. This removes the double-spiral formed by $e$ and $f$. To keep the circular order of vertices and therefore weak isomorphism, we also move each vertex between the old and the new position of $v_{d}$ on the inner circle in counter-clockwise direction into the wedge~$A$.

\begin{figure}[h]
\centering
\subcaptionbox{\centering\label{fig:double-spirals_a}}[.328\textwidth]{\includegraphics[page=1]{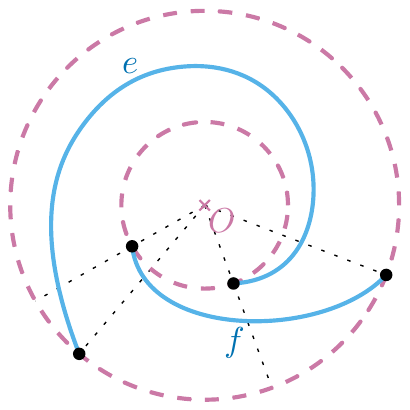}}
\subcaptionbox{\centering\label{fig:double-spirals_b}}[.328\textwidth]{\includegraphics[page=2]{Figures/special_cylindrical.pdf}}
\subcaptionbox{\centering\label{fig:double-spirals_c}}[.328\textwidth]{\includegraphics[page=3]{Figures/special_cylindrical.pdf}}
\caption{\textbf{(a)}~Two edges $e$ and $f$ forming a clockwise double-spiral. \textbf{(b)}~Moving vertex $v_{d}$, and potentially other vertices~$v_{d}'$, to resolve the clockwise double-spiral formed by $e$ and $f$. This cannot create any new counter-clockwise double-spirals. \textbf{(c)}~A clockwise (yellow) and a counter-clockwise (darkorange) double-spiral in a cylindrical drawing of $K_{8}$; circle edges are omitted for convenience.}
\label{fig:double-spirals}
\end{figure}

It remains to show that no new double-spirals are created in the process. Since we only move vertices on the inner circle in counter-clockwise direction, we cannot create any new clockwise double-spirals in the process. Moreover, every new counter-clockwise double-spiral must involve a vertex $v_{d}'$ that is moved. So let $f' = \{ v_{c}', v_{d}' \}$ be an edge with $\omega_{f'} > 0$ after~$v_{d}'$ is moved. By simplicity with the edge $f$, $v_{c}'$ has to lie between $v_{c}$ and $v_{a}$, $v_{c}$ excluded, in counter-clockwise direction on the outer circle. Hence there exist wedges $A_{1}$ between $v_{d}'$ and $v_{a}'$ and $A_{2}$ between $v_{c}$ and $v_{c}'$ that lie in the closed complement of $\Lambda_{f'}$ and have disjoint interiors. In particular, every lateral edge $e'$ with $\omega_{e'} > 0$ passing through $A_{1}$ in counter-clockwise direction, to potentially form a double-spiral with $f'$, must start on the outer circle after $v_{c}'$ but at latest at vertex~$v_{a}$ and, to keep simplicity with the edge $e$, end before vertex $v_{b}$ on the inner circle. Since $v_{c}$ lies in $\Lambda_{e}$ and, therefore, between $v_{b}$ and $v_{c}'$ in counter-clockwise direction around $O$, $e'$~cannot pass through $A_{2}$. Consequently, $f'$ cannot be part of any counter-clockwise double-spiral.

We proceed similarly to remove counter-clockwise double-spirals, moving vertices on the inner circle in clockwise direction. In each of those steps, we reduce the total number of double-spirals by at least one. So after finitely many steps we reach a weakly isomorphic cylindrical drawing $\mathcal{D}'$ without double-spirals.
\end{proof}

We actually could have formulated \Cref{lem:double-spirals} using the \emph{strong} isomorphism. The \emph{weak} isomorphism is only due to keeping the drawing c-monotone throughout the proof, which we do for convenience because it allows for a nice definition of double-spirals. Let us also remark that even a cylindrical drawing of $K_{n}$ can contain a clockwise and a counter-clockwise double-spiral at the same time; see \Cref{fig:double-spirals_c}.

Recall for the final proof of this paper that in a strongly cylindrical drawing, by definition, all circle edges must lie in the inner or outer face. Furthermore, we say that a circle edge $e$ \emph{covers another circle edge} $e'$, on the same or the other circle, if both end-vertices of $e'$ lie in~$\Lambda_{e}$, that is, $e$ and $e'$ together could potentially cover the plane, depending on~$\Lambda_{e'}$.

\begin{theorem}\label{thm:strong-cylindrical-is-c-mon}
Let $\mathcal{D}$ be a strongly cylindrical drawing. Then there exists a strongly c-monotone drawing $\mathcal{D}'$ that is weakly isomorphic to $\mathcal{D}$.
\end{theorem}

\begin{proof}
We can assume, by \Cref{prop:cylindrical-is-c-mon,lem:double-spirals}, that $\mathcal{D}$ is c-monotone and contains no double-spirals. Hence the sub-drawing $\mathcal{D}_{\ell}$ of lateral edges is already strongly c-monotone. We construct $\mathcal{D}'$ from $\mathcal{D}_{\ell}$ by suitably adding all circle edges to it from scratch. Since the face for each circle edge is fixed by being strongly cylindrical and the vertex positions are fixed by~$\mathcal{D}_{\ell}$, all crossings are already determined up to parity. The remaining task is to determine~$\Lambda_{e}$, for each circle edge $e$, by choosing a direction around $O$, clockwise or counter-clockwise, without violating strong c-monotonicity or simplicity of $\mathcal{D}'$. Because two circle edges crossing each other twice would also mean that they together cover the plane, we can restrict our attention to strong c-monotonicity.

\begin{figure}[h]
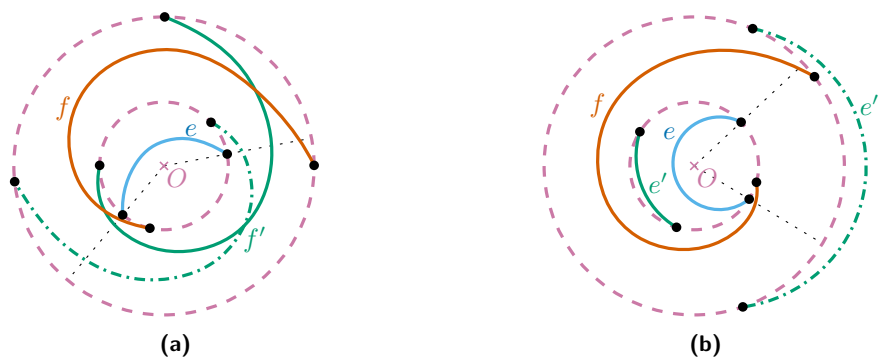

\centering
\subcaptionbox{\centering\label{fig:strongly-cylindrical-c-mon-same}}[.49\textwidth]{\includegraphics[page=4]{Figures/special_cylindrical.pdf}}
\subcaptionbox{\centering\label{fig:strongly-cylindrical-c-mon-different}}[.49\textwidth]{\includegraphics[page=5]{Figures/special_cylindrical.pdf}}
\caption{\textbf{(a)}~The lateral edge $f$ excludes one direction for $e$. Another lateral edge $f'$ cannot exclude the second direction for $e$ because otherwise $f$ and $f'$ would either cross each other twice (solid~$f'$) or form a double-spiral (dash dotted $f'$). \textbf{(b)}~The lateral edge $f$ also excludes one direction for all circle edges $e'$ covered by $e$; the dash dotted example for $e'$ is not possible.}
\label{fig:strongly-cylindrical-c-mon}
\end{figure}

We first consider possible interactions between $\Lambda_{e}$ and $\Lambda_{f}$, for a circle edge $e$ and a lateral edge $f$, once we fix one of the two directions for $e$, that is, whether $e$ together with $f$ might cover the plane. Assume there exist lateral edges $f$ and $f'$ such that one choice for the direction of $e$ leads to $e$ and $f$ covering the plane and the other choice leads to $e$ and $f'$ covering the plane. Then $f$ and $f'$ would either cross each other twice or form a double-spiral, a contradiction in any case; see \Cref{fig:strongly-cylindrical-c-mon-same} for an example. Hence, with respect to all lateral edges, the edge~$e$ has at least one of its two potential directions available.

Furthermore, if one direction of $e$ is excluded by a lateral edge $f$, then $f$ also excludes the respective direction for all circle edges that are covered by~$e$, once $e$ is drawn in its unique available direction; see \Cref{fig:strongly-cylindrical-c-mon-different} for an example. That is, two circle edges $e$ and $e'$ that each only have one direction available will not cover the plane.

Finally, if there are circle edges with both directions available, we choose some ray~$r$ starting at $O$ and use, for all those remaining edges, the direction such that they do not cross~$r$. Consequently, no pair of those edges can cover the plane. Furthermore, let $e$ be a circle edge that was forced into one direction by a lateral edge $f$ and let $e'$ be a circle edge that had both directions available. Then $e$ and $e'$ cannot cover the plane because otherwise also $f$ and $e'$ would cover the plane, a contradiction to $e'$ having both directions available. Hence this produces a strongly c-monotone drawing $\mathcal{D}'$ which is weakly isomorphic to~$\mathcal{D}$.
\end{proof}

It is essential for the proof of \Cref{thm:strong-cylindrical-is-c-mon} that the initial drawing $\mathcal{D}$ is \emph{strongly} cylindrical, because when a circle edge is drawn in the lateral face, then its direction around~$O$ is fixed from the start. In particular, we made use of that with the pretzel structure in \Cref{fig:cylindrical-example}, where we constructed a cylindrical drawing that is not strongly c-monotone.

\section{Conclusion}\label{sec:conclude}

In this paper, we extended Rafla's Hamiltonian cycle conjecture for simple drawings of $K_{n}$ (\Cref{conj:main}) to a conjecture about Hamiltonian paths between given pairs of end-vertices (\Cref{conj:stronger}). We gave an overview on known sub-classes of simple drawings (\Cref{sec:classes}) and analyzed containment relations between them (\Cref{sec:relations}). Moreover, we proved our new conjecture for several of those drawing classes, in particular, strongly c-monotone drawings and cylindrical drawings (\Cref{sec:all-pair-proofs}).

A next goal is to extend our results to more classes of simple drawings, especially, generalized twisted drawings on an even number of vertices. From there we further aim for c-monotone drawings and crossing maximal drawings. The former are of interest due to the feature we mentioned after their definition in \Cref{sec:classes}. The latter are of interest because they potentially contain only few plane sub-drawings.

Another intriguing question is to figure out the essential reason why \Cref{conj:main} should be true in general for simple drawings, while it is not true anymore for star-simple drawings. Moreover, it would be interesting to know whether the relation between \Cref{conj:stronger,conj:main} (\Cref{thm:conj-all-hp-stronger-hc}) can be formulated more specific for a single drawing instead of a whole set of drawings and whether \Cref{conj:main} also implies \Cref{conj:stronger}. We remark, however, that even if \Cref{conj:stronger} is strictly stronger in the sense that \Cref{conj:main} does not imply \Cref{conj:stronger}, it could potentially be easier to prove. Conversely, if \Cref{conj:stronger} should turn out to be false, this might give some insight on Rafla's original conjecture as well.

Finally, regarding relations between classes of simple drawings, we would like to know whether there is an inclusion relation between the rather combinatorially defined class of g-convex drawings and any of the more topologically defined classes like c-monotone drawings.

\bibliography{cgt_cfhcs}

\end{document}